\documentclass[11pt,a4paper]{amsart}
\usepackage{indentfirst,latexsym,bm}
\usepackage{amsfonts}
\usepackage{amssymb}
\usepackage{amsmath}
\usepackage{amsbsy}
\usepackage{dsfont}
\usepackage{amsthm}
\usepackage{leftidx}
\usepackage{amscd}
\usepackage[all]{xy}
\usepackage{chemarrow}
\usepackage{color}
\usepackage{multirow}
\usepackage{hyperref}
\usepackage[titletoc]{appendix}
\begin{document}
%\begin{CJK}{GBK}{song}
\title[On Hopf algebras without the dual Chevalley property]{On  Hopf algebras over the unique $12$-dimensional Hopf algebra without the dual Chevalley property}
\author{Rongchuan Xiong}
\address{School of Mathematical Sciences, Shanghai Key Laboratory of PMMP, East China Normal University, Shanghai 200241, China}
\email{rcxiong@foxmail.com}
%\thanks{The author }
%\subjclass[2010]{16T05, 16S35, 18D10}
\date{}
\maketitle

\newtheorem{question}{Question}
\newtheorem{defi}{Definition}[section]
\newtheorem{conj}{Conjecture}
\newtheorem{thm}[defi]{Theorem}
\newtheorem{lem}[defi]{Lemma}
\newtheorem{pro}[defi]{Proposition}
\newtheorem{cor}[defi]{Corollary}
\newtheorem{rmk}[defi]{Remark}
\newtheorem{Example}{Example}[section]

\theoremstyle{plain}
\newcounter{maint}
\renewcommand{\themaint}{\Alph{maint}}
\newtheorem{mainthm}[maint]{Theorem}

\theoremstyle{plain}
\newtheorem*{proofthma}{Proof of Theorem A}
\newtheorem*{proofthmb}{Proof of Theorem B}
\newcommand{\tabincell}[2]{\begin{tabular}{@{}#1@{}}#2\end{tabular}}

\newcommand{\C}{\mathcal{C}}
\newcommand{\D}{\mathcal{D}}
\newcommand{\A}{\mathcal{A}}
\newcommand{\De}{\Delta}
\newcommand{\M}{\mathcal{M}}
\newcommand{\K}{\mathds{k}}
\newcommand{\E}{\mathcal{E}}
\newcommand{\Pp}{\mathcal{P}}
\newcommand{\Lam}{\lambda}
\newcommand{\As}{^{\ast}}
\newcommand{\Aa}{a^{\ast}}
\newcommand{\Ab}{(a^2)^{\ast}}
\newcommand{\Ac}{(a^3)^{\ast}}
\newcommand{\Ad}{(a^4)^{\ast}}
\newcommand{\Ae}{(a^5)^{\ast}}
\newcommand{\B}{b^{\ast}}
\newcommand{\BAa}{(ba)^{\ast}}
\newcommand{\BAb}{(ba^2)^{\ast}}
\newcommand{\BAc}{(ba^3)^{\ast}}
\newcommand{\BAd}{(ba^4)^{\ast}}
\newcommand{\BAe}{(ba^5)^{\ast}}

\newcommand{\CYD}{{}^{\C}_{\C}\mathcal{YD}}
\newcommand{\DM}{{}_{D}\mathcal{M}}
\newcommand{\BN}{\mathcal{B}}
\newcommand{\HYD}{{}^{H}_{H}\mathcal{YD}}
\newcommand{\Ga}{g^{\ast}}
\newcommand{\Gb}{(g^2)^{\ast}}
\newcommand{\Gc}{(g^3)^{\ast}}
\newcommand{\Gd}{(g^4)^{\ast}}
\newcommand{\Ge}{(g^5)^{\ast}}
\newcommand{\X}{x^{\ast}}
\newcommand{\GXa}{(gx)^{\ast}}
\newcommand{\GXb}{(g^2x)^{\ast}}
\newcommand{\GXc}{(g^3x)^{\ast}}
\newcommand{\GXd}{(g^4x)^{\ast}}
\newcommand{\GXe}{(g^5x)^{\ast}}
\newcommand{\Z}{\mathds{Z}}
\newcommand{\G}{\mathcal{G}}
\newcommand{\I}{\mathds{I}}
\newcommand\ad{\operatorname{ad}}
\newcommand{\Alg}{\Hom_{\text{alg}}}
\newcommand\Aut{\operatorname{Aut}}
\newcommand{\AuH}{\Aut_{\text{Hopf}}}
\newcommand\coker{\operatorname{coker}}
\newcommand\car{\operatorname{char}}
\newcommand\Der{\operatorname{Der}}
\newcommand\diag{\operatorname{diag}}
\newcommand\End{\operatorname{End}}
\newcommand\id{\operatorname{id}}
\newcommand\gr{\operatorname{gr}}
\newcommand\Soc{\operatorname{Soc}}
\newcommand\Top{\operatorname{Top}}
\newcommand\GK{\operatorname{GKdim}}
\newcommand{\Hom}{\operatorname{Hom}}
\newcommand\ord{\operatorname{ord}}
\newcommand{\N}{\mathds{N}}
\begin{abstract}
Let $\mathds{k}$ be an algebraically closed field of characteristic zero. We determine all finite-dimensional Hopf algebras over $\mathds{k}$ whose Hopf coradical is isomorphic to the unique $12$-dimensional Hopf algebra $\mathcal{C}$ without the dual Chevalley property, such that the diagrams are strictly graded and the corresponding infinitesimal braidings are indecomposable objects in ${}_{\mathcal{C}}^{\mathcal{C}}\mathcal{YD}$.
In particular, we obtain new Nichols algebras of dimension $18$ and $36$ and two families of new Hopf algebras of dimension $216$.

\end{abstract}

\section{Introduction}

Let $\mathds{k}$ be an algebraically closed field of characteristic zero. This work contributes to the classification of finite-dimensional Hopf algebras over $\mathds{k}$ without the dual Chevalley property, that is, the coradical is
not a subalgebra. Until now, there are few classification results on such Hopf algebras with non-pointed duals, some exceptions being \cite{GG16,HX17}.

The strategy follows the principle proposed by Andruskiewitsch and Cuadra \cite{AC13}, that is, the so-called generalized lifting method.
Let $A$ be a Hopf algebra over $\K$ without the dual Chevalley property and $A_{[0]}$ the subalgebra generated by the coradical $A_0$ of $A$.  We will say $A_{[0]}$ is the Hopf coradical of $A$. As a generalization of the lifting method \cite{AS98}, the idea of the generalized lifting method is to replace the coradical filtration  with the standard filtration $\{A_{[n]}\}_{n\geq 0}$, which is
defined recursively by $A_{[n]}=A_{[n-1]}\bigwedge A_{[0]}$.
Assume that $S_A(A_{[0]})\subseteq A_{[0]}$, the filtration $\{A_{[n]}\}_{n\geq 0}$ is a Hopf algebra filtration, which implies that  the associated graded coalgebra $\gr A=\oplus_{n=0}^{\infty}\gr^nA$ is a Hopf algebra, where $\gr^nA=A_{[n]}/A_{[n-1]}$ and $A_{[-1]}=0$.
It follows by \cite[Theorem 2]{R85} that there exists uniquely a connected graded braided Hopf algebra $R=\oplus_{n=0}^{\infty}R(n)$  in ${}^{A_{[0]}}_{A_{[0]}}\mathcal{YD}$ such that $\gr A\cong R\sharp A_{[0]}$.  Moreover,  $R(n)=R\cap\gr^nA$ and $R(1)\subset\Pp(R)$. We call $R$ or $R(1)$ the diagram or infinitesimal braiding of $A$, respectively.
If  $A_0$ is a Hopf subalgebra, then the standard filtration coincides with the coradical filtration. $\gr A$ is coradically graded and $R$ is \emph{strictly graded}, that is, $R(0)=\K$, $R(1)=\Pp(R)$.  In general, it is an open question whether the diagram $R$ is strictly graded. See \cite{AS02,AC13} for details.

 Now we outline the strategy. We first fix a finite-dimensional Hopf algebra $\C$, which is generated
 by its coradical $\C_0$. Then we determine those $V\in {}_{\C}^{\C}\mathcal{YD}$ such that $\dim\BN(V)<\infty$
 and present $\BN(V)$ by generators and relations  using the skew-derivation. Finally, we calculate  all possible Hopf algebras $A$  such that $\gr A=\BN(V)\sharp \C$. It should be pointed out that the generalized lifting method was firstly used by G.-A.~Garcia and J.-M.-J.~Giraldi \cite{GG16} to construct new examples of Hopf algebras of dimension $64$.

The present paper is a sequel to \cite{AC13,GG16,HX16}. In \cite{HX16}, the authors fixed a Hopf algebra $\C$ $($See Definition \ref{proStrucOfC}$)$
and constructed some Hopf algebras of dimension $72$ without
the dual Chevalley property but left some questions about Nichols algebras unsolved. Denote by $\D$ the Drinfeld double of $\C^{cop}$. In this paper, we first describe all simple and indecomposable projective $\D$-modules and calculate the decompositions of tensor products of them. As a consequence, we obtain the projective class
ring $r_p(\D)$, that is, $r_p(\D)\cong\Z[y_0,\ldots,y_5]/J$, where
\begin{align*}
J=(y_0^6=1,\ y_iy_3^2=2y_iy_0+y_iy_0^2+y_i,\  y_jy_{6-j}=y_3^2,\ y_ky_l=y_{k+l}+y_{k+l}y_0),
\end{align*}
for $i,j,k,l,k+l\in\I_{1,5}$. Moreover, $\D$ is of wild type. See section \ref{secPresentation} for details.

 From Theorem \ref{thmsimplemoduleD}, the simple objects in $\CYD$ consist of $6$ one-dimensional objects $\K_{\chi^{k}}$ for $k\in \I_{0,5}$ and $30$ two-dimensional objects $V_{i,j}$ for $(i,j)\in \Lambda$, where $\Lambda=\{(i,j)\in \I_{0,5}\times \I_{0,5}\mid 3i\neq j\}$. These braided vector spaces $V_{i,j}$  have already appeared in \cite{Hi93,AGi17}. More precisely,
 $V_{i,0}$ and $V_{i,3}$  belong to the case $\mathfrak{R}_{2,1}$ in \cite{Hi93,AGi17}, and the others belong to the case $\mathfrak{R}_{1,2}$. In particular, as stated in \cite{AGi17}, they are not of diagonal type.

Next, we study Nichols algebras over the indecomposable objects in $\CYD$ and present explicitly the Nichols algebras of finite dimension. We obtain the following result:
\begin{mainthm}\label{thmA}
 Let $\BN(V)$ be a finite-dimensional Nichols algebra over an indecomposable object $V$ in $\CYD$. Then $V$ must be simple and isomorphic
 either to $\K_{\chi^{k}}$ for $k\in\{1,3,5\}$, $V_{1,1}$, $V_{4,2}$, $V_{3,1}$, $V_{2,2}$, $V_{1,4}$, $V_{4,5}$, $V_{2,4}$,
 $V_{3,5}$, $V_{4,4}$, $V_{1,5}$, $V_{4,1}$ or $V_{1,2}$. Moreover, the generators and relations of $\BN(V)$ are given by the following table: \\
%\begin{table}
\center
%\caption{Finite-dimensional Nichols algebras}
\begin{tabular}{|c|c|c|}
  \hline
  $V$           &   relations of $\BN(V)$ with generators $x,y$         & $\dim\BN(V)$ \\\hline
$\K_{\chi^1},l=1,3,5$   &   $x^2=0$                                       &   2   \\\hline
$V_{3,j},j=1,5$       &  $ x^2=0,\  xy+\xi^{-j}yx=0,\  y^3=0$     &  6  \\\hline
%$V_{3,5}$       &  $ x^2=0, xy-\xi^4yxv_1=0, y^3=0 $   &  6  \\\hline
%$V_{2,2}$       &  $ x^2-\xi^2y^2=0, xy-yx=0, x^3=0 $    &  6  \\\hline
$V_{2,j+3},j=1,5$       &   $y^2+\xi x^2=0,\  xy-yx=0,\  x^3=0  $   &  6  \\\hline
$V_{1,j},j=1,5$, &\tabincell{c}{ $x^6=0$, $x^2y+\xi^jyx^2+(1+\xi^j)xyx=0$,\\$x^3+y^2x+xy^2+yxy=0$, $x^2y+yx^2+xyx+y^3=0$ }& 36
\\\hline
$V_{4,j},j=1,5$  &\tabincell{c}{ $x^3=0$, $y^3-x^2y-yx^2+xyx=0$,\\$y^2x+xy^2-yxy=0$, $\xi^jx^2y+\xi^{5j}yx^2+xyx=0$ } &    18
\\\hline
$V_{4,j+3},j=1,5$&\tabincell{c}{ $x^3=0$, $\xi^{2j}x^2y+\xi^{4j}xyx+yx^2=0$,\\$y^6=0$, $y^2x+(\xi^{5j}+\xi^{4j})yxy-xy^2=0$ } & 36          \\\hline
$V_{1,j+3},j=1,5$&\tabincell{c}{ $x^3=0$, $y^3+x^2y+yx^2+xyx=0$,\\$y^2x+xy^2+yxy=0$, $\xi^{2j}x^2y+\xi^{4j}xyx+yx^2=0$} &       18              \\\hline
\end{tabular}
%\end{table}
\end{mainthm}
The Nichols algebras $\BN(\K_{\chi^l})$ with $l\in\{1,3,5\}$ are exterior algebras. The Nichols algebras $\BN(V_{2,j+3})$ and $\BN(V_{3,j})$ for $j\in\{1,5\}$  are isomorphic to quantum planes as algebras but not as coalgebras. The Nichols algebras of dimension $6$ in Theorem \ref{thmA} were firstly introduced in \cite{HX16} and also appeared in \cite{AGi17}. The rest of the Nichols algebras do not admit quadratic relations. As far as we know, they constitute new examples.
See section \ref{secNicholsalg} for details.

 Finally, we study the deformations of the bosonizations of these Nichols algebras. We obtain the following result:
 \begin{mainthm}\label{thmB}
Let $A$ be a finite-dimensional Hopf algebra over $\C$ such that the corresponding infinitesimal braiding is an indecomposable object $V$ in $\CYD$. Assume that the diagram of $A$ is strictly graded. Then $V$ is simple and $A$ is isomorphic either to
\begin{itemize}
  \item[(a)] $\bigwedge\K_{\chi^{k}}\sharp \C$ for $k\in\{1,3,5\}$;
  \item[(b)] $\BN(V_{3,1})\sharp\C$; $\BN(V_{3,5})\sharp\C$; $\BN(V_{2,2})\sharp\C$; $\BN(V_{2,4})\sharp\C$;
  \item[(c)] $\BN(V_{4,1})\sharp\C$; $\BN(V_{4,5})\sharp\C$;
  \item[(d)] $\BN(V_{1,1})\sharp\C$; $\BN(V_{1,5})\sharp\C$; $\BN(V_{4,2})\sharp\C$; $\BN(V_{4,4})\sharp\C$;
  \item[(e)] $\mathfrak{B}_{1,2}(\mu)$ for some $\mu\in\K$;
  \item[(f)] $\mathfrak{B}_{1,4}(\mu)$ for some $\mu\in\K$.
\end{itemize}
\end{mainthm}

The Hopf algebras in items $(a)-(d)$ are basic Hopf algebras of dimension $24$, $72$, $216$ and $432$, respectively. The Hopf algebras in items $(e)-(f)$ introduced in Definition \ref{defiB12} have dimension $216$ without the dual Chevalley property. Moreover, they do not have pointed duals and constitute new examples of Hopf algebras except for $\mu=0$. See section \ref{secHopfalgebra} for details.

The paper is organized as follows:
In section $\ref{Preliminary}$, we introduce notations about Yetter-Drinfeld modules, Nichols algebras and recall useful results  in \cite{HX16}.
In section $\ref{secPresentation}$, we study the projective class ring and representation type of $\D$.
 In section $\ref{secNicholsalg}$, we determine all finite-dimensional Nichols algebras over the indecomposable objects in $\CYD$ and present them by generators and relations. In section $\ref{secHopfalgebra}$, we determine all finite-dimensional
Hopf algebras  over $\C$ whose diagrams are strictly graded and the infinitesimal braidings are simple objects in $\CYD$.

\section{Preliminaries}\label{Preliminary}
\paragraph{Conventions.} Our ground field is an algebraically closed field $\K$ of characteristic zero. We denote by $\xi$ a primitive $6$-th root of unity and $\Lam=(\xi-1)(\xi+1)^{-1}$. The references for Hopf algebras are \cite{M93,R11}.

The notation for a Hopf algebra $H$ over $\K$ is standard: $\Delta$, $\epsilon$, and $S$ denote the comultiplication, the counit and the antipode. We use Sweedler's notation for the comultiplication and coaction.
%Given a Hopf algebra $H$ with bijective antipode, we denote by $H^{op}$ the Hopf algebra with the opposite multiplication, by $H^{cop}$ the Hopf algebra with the opposite comultiplication, and by $H^{bop}$ the Hopf algebra $H^{op\,cop}$.
Denote by $\G(H)$ the set of group-like elements of $H$. For $g,h\in\G(H)$,  $\Pp_{g,h}(H)=\{x\in H\mid \De(x)=x\otimes g+h\otimes x\}$. In particular, the linear space $\Pp(H):=\Pp_{1,1}(H)$ is called the set
 of primitive elements.
% If $V$ is a $\K$-vector space, $v\in V$ and $f\in V\As$, we use either $f(v)$, $\langle f$, $v\rangle$, or $\langle v,f\rangle$ to denote the evaluation.
Given $n\geq k\geq 0$ and $q\in\K$, we denote $\Z_n=\Z/n\Z$, $\I_{k,n}=\{k,k+1,\ldots,n\}$ and $(n)_q=\sum_{j=0}^{n-1}q^j$.
% In Particular, The Operations $Ij$ And $I\Pm J$ Are Considered Modulo $N+1$ For $I,J\In\I_{K,N}$ When Not Specified.
If $M$ is a left $H$-module, we denote by $\Top(M)$ the top of $M$, by $\Soc(M)$  the socle of $M$, by $ExtQ(M)$ the Gabriel quiver of $M$, and by $\Pp(M)$ the projective cover of $M$.

\subsection{ Yetter-Drinfeld modules and Nichols algebras.}
Let $H$ be a Hopf algebra with bijective antipode. A left \emph{Yetter-Drinfeld module} $M$ over $H$ is a left $H$-module $(M,\cdot)$ and a left $H$-comodule $(M, \delta)$ satisfying
\begin{align*}
\delta(h\cdot v)=h_{(1)}v_{(-1)}S(h_{(3)})\otimes h_{(2)}\cdot v_{(0)},\quad \forall v\in V,\ h\in H.
\end{align*}

Let ${}^{H}_{H}\mathcal{YD}$ be the
category of Yetter-Drinfeld modules over $H$. %with $H$-linear and $H$-colinear maps as morphisms.
Then ${}^{H}_{H}\mathcal{YD}$ is braided monoidal. For $V,W\in {}^{H}_{H}\mathcal{YD}$, the braiding $c_{V,W}$ is given by
\begin{align}\label{equbraidingYDcat}
c_{V,W}:V\otimes W\mapsto W\otimes V,\ v\otimes w\mapsto v_{(-1)}\cdot w\otimes v_{(0)},\ \forall\,v\in V, w\in W.
\end{align}
In particular, $(V,c_{V£¬V})$ is a braided vector space, that is, $c:=c_{V,V}$ is a linear isomorphism satisfying the
braid equation $(c\otimes\text{id})(\text{id}\otimes c)(c\otimes\text{id})=(\text{id}\otimes c)(c\otimes\text{id})(\text{id}\otimes c)$. Moreover, ${}^{H}_{H}\mathcal{YD}$ is rigid. The left dual $V\As$ is defined by
\begin{align}\label{eqDual-YD}
\langle h\cdot f,v\rangle=\langle f,S(h)v\rangle,\quad f_{(-1)}\langle f_{(0)},v\rangle=S^{-1}(v_{(-1)})\langle f, v_{(0)}\rangle.
%\langle h\cdot f,v\rangle=\langle f,S^{-1}(h)v\rangle,\quad f_{(-1)}\langle f_{(0)},v\rangle=S(v_{(-1)})\langle f, v_{(0)}\rangle.
\end{align}

If $H$ is finite-dimensional,  then by \cite[Proposition\,2.2.1.]{AG99}, $\HYD\cong{}_{H^{\ast}}^{H^{\ast}}\mathcal{YD}$ as braided monoidal categories  via the functor $(F,\eta)$ defined as follows:
 $F(V)=V$ as a vector space,
\begin{align}
\begin{split}\label{eqVHD}
f\cdot v=f(S(v_{(-1)}))v_{(0)},\quad \delta(v)=\sum_{i}S^{-1}(h^i)\otimes h_i\cdot v,\  \text{and}\\
\eta:F(V)\otimes F(W)\mapsto F(V\otimes W), v\otimes w\mapsto w_{(-1)}\cdot v\otimes w_{(0)}
\end{split}
\end{align}
for $V,W\in\HYD$, $f\in H^{\ast},v\in V, w\in W$. Here $\{h_i\}$ and $\{h^i\}$ are the dual bases of $H$ and $H\As$.

An algebra $(R,m,\mu)$ in $\HYD$ is an algebra  such that $R\in\HYD$ and the multiplication $m$ and unit $\mu$ are morphisms in $\HYD$. Same for a coalgebra $(R,\Delta,\epsilon)$ in $\HYD$. A bialgebra $(R,m,\mu,\Delta,\epsilon)$ in $\HYD$ is an algebra $(R,m,\mu)$ and a coalgebra $(R,\Delta,\epsilon)$ in $\HYD$ such that $\Delta:R\rightarrow R\otimes R$ and $\epsilon:R\rightarrow\K$ are algebra maps in $\HYD$, that is,
\begin{gather}\label{eqBriadingHopfalgebra}
\Delta m=(m\otimes m)(\id_R\otimes c_{R,R}\otimes \id_R)(\Delta\otimes\Delta),\quad \epsilon(ab)=\epsilon(a)\epsilon(b), \ \forall a,b\in R.
\end{gather}
If moreover there is a morphism $S:R\rightarrow R$ in $\HYD$ such that $m(S\otimes\id)\Delta=\mu\epsilon=m(\id\otimes S)\Delta$, then we say $R$ is a Hopf algebra in $\HYD$. See for example \cite{AG99} for details and more references.
\begin{defi}\cite[Definition\;2.1.]{AS02}\label{defiNichols}
Let $H$ be a Hopf algebra and $V\in{}^{H}_{H}\mathcal{YD}$. The Nichols algebra $\BN(V)$ over $V$ is a $\N$-graded Hopf algebra $R=\oplus_{n\geq 0} R(n)$ in ${}^{H}_{H}\mathcal{YD}$ such that
\begin{align*}
R(0)=\mathds{k}, \quad R(1)=V,\quad
R\;\text{is generated as an algebra by}\;R(1),\quad
\Pp(R)=V.
\end{align*}
\end{defi}
  Let $V\in\HYD$. The tensor algebra $T(V)=\oplus_{n=0}^{\infty}T^n(V)$ is a $\N$-graded Hopf algebra in $\HYD$ with the comultiplication given by $\Delta(v)=v\otimes 1+1\otimes v,\ v\in V$. $\BN(V)$ is  isomorphic to $T(V)/I(V)$, where $I(V)=\oplus_{n\geq 2}I^n(V)$ is the largest $\mathds{N}$-graded ideal and coideal of $T(V)$ in $\HYD$ such that $I(V)\cap V=0$. Moreover, $\BN(V)$ as a coalgebra and an algebra depends only on $(V,c)$.
\begin{rmk}\label{rmkN-infity}
Let  $W$ be a vector subspace of $(V, c)$ such that $c(W\otimes W)\subset W\otimes W$. Then $\dim \BN(V)=\infty$ if $\dim \BN(W)=\infty$. This
occurs for example, when $V$ contains a non-zero element $v$ such that $c(v\otimes v)=v\otimes v$.
 See  \cite{G00} for details.
\end{rmk}

Now we recall the standard tool, the so called skew-derivation, for working with Nichols algebras. Let $(V,c)$ be a $n$-dimensional $($rigid$)$ braided vector space and $\Delta^{i,m-i}:T^m(V)\rightarrow T^{i}(V)\otimes T^{m-i}(V)$ the $(i,m-i)$-homogeneous component of the comultiplication $\Delta:T(V)\rightarrow T(V)\otimes T(V)$ for $m\in\N$ and $i\in\I_{0,m}$. Given $f\in V^{\ast}$, the \emph{skew-derivation} $\partial_f\in\text{End}\,T(V)$ is given by
\begin{gather}\label{eqSkew-1}
\partial_f(v)=(f\otimes \id)\Delta^{1,m-1}(v):T^m(V)\rightarrow T^{m-1}(V),\quad v\in T^m(V),\ m\in\N.
\end{gather}
Let $\{v_i\}_{1\leq i\leq n}$ and $\{v^i\}_{1\leq i\leq n}$ be the dual bases of $V$ and $V\As$. We write $\partial_i:=\partial_{v^i}$ for simplicity. The skew-derivation is very useful
to find the relations of the Nichols algebra $\BN(V)$:
\begin{align}\label{Def-Nichols-IV}
I^m(V)=\{r\in T^m(V)\mid \forall f_1, f_2,\ldots,f_m\in V\As,\  \partial_{f_1}\partial_{f_2}\cdots\partial_{f_m}(r)=0\}.
\end{align}
Then  $\partial_f$ for any $f\in V\As$ can be extended to $\BN(V)$ and $\mathop\cap\limits_{f\in V\As}\ker\partial_f=\K$ in $\BN(V)$. See for example \cite{AS02,AHS10} for details and more references.

 We close this subsection by giving the explicit relation between $V$ and $V\As$ in ${}^{H}_{H}\mathcal{YD}$.
\begin{pro}\cite[Proposition\,3.2.30]{AG99}\label{proNicholsdual}
Let $V$ be an object in ${}_H^H\mathcal{YD}$. If $\BN(V)$ is finite-dimensional, then $\BN(V\As)\cong \BN(V)^{\ast\,bop}$.
\end{pro}

\subsection{Bosonization and Hopf algebras with a projection.}In this subsection, we follow \cite{R85}.
Let $R$ be a Hopf algebra in ${}^{H}_{H}\mathcal{YD}$. We write $\Delta_R(r)=r^{(1)}\otimes r^{(2)}$ to avoid confusions. The bosonization $R\sharp H$ is defined as follows:
$R\sharp H=R\otimes H$ as a vector space, and the multiplication and comultiplication are given by the smash product and smash coproduct, respectively:
\begin{align}
(r\sharp g)(s\sharp h)=r(g_{(1)}\cdot s)\sharp g_{(2)}h,\quad
\Delta(r\sharp g)=r^{(1)}\sharp(r^{(2)})_{(-1)}g_{(1)}\otimes (r^{(2)})_{(0)}\sharp g_{(2)}.\label{eqSmash}
\end{align}
Clearly, the map $\iota:H\rightarrow R\sharp H, h\mapsto 1\sharp h,\ \forall h\in H$ is injective and the map
$\pi:R\sharp H\rightarrow H,r\sharp h\mapsto \epsilon_R(r)h,\ \forall r\in R, h\in H$ is surjective such that $\pi\circ\iota=\id_H$.
Moreover, $R=(R\sharp H)^{coH}=\{x\in R\sharp H\mid (\text{id}\otimes\pi)\Delta(x)=x\otimes 1\}$.

Conversely, if $A$ is a Hopf algebra and $\pi:A\rightarrow H$ is a bialgebra map admitting a bialgebra
section $\iota:H\rightarrow A$ such that $\pi\circ\iota=\id_H$, then $A\simeq R\sharp H$, where $R=A^{coH}$ is a Hopf algebra in ${}^{H}_{H}\mathcal{YD}$ with the Yetter-Drinfeld module structure and comultiplication given by
\begin{gather}\label{eQ-sh}
h\cdot r=h_{(1)}rS_A(h_{(2)}),\quad
\delta(r)=(\pi\otimes id)\Delta_A(r),\quad
\Delta_R(r)=r_{(1)}(\iota S_H(\pi(r_{(2)})))\otimes r_{(3)}.
%\epsilon_R=\epsilon_A|_R,\;
%S_R(r)=(\iota\pi(r_{(1)}))S_A(r_{(2)}).
\end{gather}

\subsection{The Hopf algebra $\C$ and the Drinfeld double $\D$}
We recall  useful results appearing in \cite{HX16}. All $12$-dimensional Hopf algebras over $\K$ were classified
by Natale \cite{Na02}. It turns out that these non-semisimple Hopf algebras are pointed, except for one case
defined as follows:
\begin{defi}\cite[Proposition\,3.4.]{HX16}\label{proStrucOfC}
Let $\C$ be the Hopf algebra generated as an algebra by the elements $a$ and $b$ satisfying the relations
\begin{align}
a^6=1, \quad b^2=0,\quad ba=\xi ab; \label{eqDef1}
\end{align}
with the  coalgebra structure  given by
\begin{align}
 \De(a)=a\otimes a+ (\xi^4+\xi^5)b\otimes ba^3,\quad\De(b)=b\otimes a^4+a\otimes b,\quad \epsilon(a)=1,\quad\epsilon(b)=0;\label{eqDef}
\end{align}
and the antipode given by $S(a)=a^5$ and $S(b)=\xi^{-2}ba$.
\end{defi}

\begin{rmk}\label{rmk-C}
%$\C$ is generated by a simple coalgebra $C:=\K\{a,ba^3,a^4,b\}$.
Up to isomorphism, $\C$ is the only Hopf algebra of dimension $12$ without the dual Chevalley property. It should be mentioned that the non-pointed basic Hopf algebra of dimension $8$ is the only Hopf algebra of dimension less than $12$ without the dual Chevalley property. See for example \cite{Na02, BG13} for details.
\end{rmk}
The following remark follows by a direct computation using the Hopf algebra structure of $\C$.

\begin{rmk}\label{rmkDDDDD}
\begin{enumerate}
\item  $\G(\C)=\K\{1,a^3\}$, $\Pp_{1,a^3}(\C)=\K\{ba^2,1-a^3\}$. $\{a^j, ba^j,\ j\in\I_{0,5}\}$ is a linear basis of $\C$.

\item Let $\widetilde{x}=\sum_{i=0}^{5}(ba^i)\As$ and $\widetilde{g}=\sum_{i=0}^{5}\xi^{-i}(a^i)\As$. Here $\{(a^j)\As,(ba^j)\As,\ j\in\I_{0,5}\}$ is the basis of $\C\As$ dual to $\{a^j,ba^j,\ j\in\I_{0,5}\}$. Then
\begin{align*}
\Delta(\widetilde{x})=\widetilde{x}\otimes \epsilon+\widetilde{g}\otimes \widetilde{x},\quad
\Delta(\widetilde{g})=\widetilde{g}\otimes \widetilde{g}.
\end{align*}

%\item Let $\alpha\in G(\C\As)=Alg(\C,\K)$. Since $a^6=1,b^2=0,ba=\xi ab$, it follows that $\alpha(a)$ is a $6$-th root of unity and $\alpha(b)=0$. Thus
%    \begin{align*}
%    G(\C\As)=\{\alpha_i=1\As{+}\xi^{-i}\Aa{+}\xi^{-2i}\Ab{+}\xi^{-3i}\Ac{+}\xi^{-4i}\Ad{+}\xi^{-5i}\Ae\}.
%  \end{align*}
%   In particular, $\alpha_0=\epsilon$, $\alpha_i=(\alpha_1)^i$, and $G(\C\As)\simeq Z_6$ with generator $\alpha_1$ or $\alpha_5$.
\item  Let $\A_1$ be the pointed Hopf algebra generated by $g$, $x$ satisfying the relations $g^6=1$, $x^2=1-g^2$ and $gx=-xg$ with  $\De(g)=g\otimes g$ and $\De(x)=x{\otimes} 1+g{\otimes} x$. It is a Radford algebra \cite{R75}. A linear basis of $\A_1$ is given by $\{g^i,g^ix,\ i\in\I_{0,5}\}$. Moreover, $\A_1\cong\C\As$ and the  Hopf algebra isomorphism $\phi:\A_1\mapsto \C\As$ is given by
\begin{align*}
\phi(g^i)&=\sum_{j=0}^5\xi^{-ij}(a^j)\As,\quad \phi(g^ix)=\theta\sum_{j=0}^5\xi^{-(j+1)i}(ba^j)\As,\quad\text{where } \theta^2=\xi^2.
%\phi(g^i)&=1\As+\xi^{-i}\Aa+\xi^{-2i}\Ab+\xi^{-3i}\Ac+\xi^{-4i}\Ad+\xi^{-5i}\Ae,\\
%\phi(g^ix)&=\theta(\xi^{-i}\B+\xi^{-2i}\BAa+\xi^{-3i}\BAb+\xi^{-4i}\BAc+\xi^{-5i}\BAd+\BAe),
\end{align*}
%Here $\{ g^i, g^ix\}_{i\in\I_{0,5}}$  is a linear basis of $\A_1$.
\end{enumerate}
\end{rmk}

We end up this subsection by describing the structure of the Drinfeld double $\D:=\D(\C^{cop})$. Recall that  $\D(\C^{cop})=\C^{*op\,cop}\otimes \C^{cop}$ is a Hopf algebra with the tensor product coalgebra structure and the algebra structure given by
$(p\otimes a)(q\otimes b)=p\langle q_{(3)}, a_{(1)}\rangle q_{(2)}\otimes a_{(2)}\langle q_{(1)}, S^{-1}(a_{(3)})\rangle$.
\begin{pro}\cite[Proposition\,3.10.]{HX16}
$\D:=\D(\C^{cop})$ as a coalgebra is isomorphic to the tensor coalgebra $\A_1^{op\,cop}\otimes \C^{cop}$, and as an algebra is generated by the elements $a$, $b$, $g$, $x$ satisfying the relations in $\C^{cop}$, the relations in $\A_1^{op\,cop}$ and
\begin{gather*}
ag=ga,\quad ax+\xi^{-2}xa=\Lam^{-1}\theta\xi^{-2}(ba^3-gb),\quad
bg=-gb,\quad bx+\xi^{-2}xb=\theta\xi^{-2}(a^4-ga).
\end{gather*}
\end{pro}

\section{The projective class ring and representation type of the Drinfeld double $\D$}\label{secPresentation}
We study the representation type of $\D$ and the projective class ring of $\D$, which is a subring of the
Green ring. We refer to \cite{ARS95} for the representation theory.
\subsection{The projective class ring of $\D$}
Recall that the \emph{Green ring}  $r(\D)$ of $\D$ can be defined as follows: $r(\D)$ as an abelian group is generated by the isomorphism classes $[V]$ of $V\in{}_{\D}\mathcal{M}$ modulo the relations $[V\oplus W]=[V]+[W]$ and its
multiplication is given by the tensor product in ${}_{\D}\mathcal{M}$, that is, $[V][W]=[V\otimes W]$. The \emph{projective class ring} $r_p(\D)$
as a subring of $r(\D)$ is generated by simple modules and projective modules. The investigation of projective class ring and green ring has received enormous attention as they are important to study the monoidal structure of the category of modules over a Hopf algebra. See fox example \cite{CMLS17} for details and more references.
We first describe the simple and indecomposable projective $\D$-modules.

\begin{defi}$\label{onesimple}$
Let $i\in\I_{0,5}$ and $\chi$ be an irreducible character of the cyclic group $\Z_6$. Denote by $\K_{\chi^{i}}$ the one-dimensional
left $\D$-module defined by
\begin{align*}
\chi^i(a)=\xi^i, \quad \chi^i(b)=0,\quad \chi^i(g)=(-1)^i,\quad \chi^i(x)=0.
\end{align*}
\end{defi}
\begin{defi}$\label{twosimple}$
For $(i,j)\in\Lambda=\{(i,j)\in \I_{0,5}\times \I_{0,5}\mid 3i\neq j\}$, let $V_{i,j}$ be
the $2$-dimensional left $\D$-module whose matrices defining $\D$-action with respect to a fixed basis are of the form:
\begin{align*}
    [a]_{i,j}&=\left(\begin{array}{ccc}
                                   \xi^i & 0\\
                                   0 &    \xi^{i+1}
                                 \end{array}\right),\quad
    [b]_{i,j}=\left(\begin{array}{ccc}
                                   0 & 1\\
                                   0 & 0
                                 \end{array}\right),\quad
    [g]_{i,j}=\left(\begin{array}{ccc}
                                   \xi^j & 0\\
                                   0 & -\xi^j
                                 \end{array}\right),\\
    [x]_{i,j}&=\left(\begin{array}{ccc}
                                   0 & \theta^{-1}\xi^{2-i}(\xi^{3i}+\xi^j)\\
                                   \theta\xi^{i-2}(\xi^{3i}-\xi^j) & 0
                                 \end{array}\right).
\end{align*}

\end{defi}

\begin{thm}\cite[Theorem\,4.4.]{HX16}\label{thmsimplemoduleD}
There exist exactly $36$ pairwise non-isomorphic simple left $\D$-modules, among which $6$ one-dimensional modules are given in Definition
$\ref{onesimple}$ and $30$ two-dimensional simple modules are given in Definition $\ref{twosimple}$.
\end{thm}

Recall that if $V$ is a simple $\D$-module, then $\Pp(V)$ is unique $($up to isomorphism$)$ indecomposable projective $\D$-module, which maps onto $V$. Let $\text{Irr}(\D)$ be the set of isomorphism classes of simple $\D$-modules. Then
$\D\cong \oplus_{V\in\text{Irr}(D)}\Pp(V)^{\dim V}$. See for example \cite{ARS95} for details.

 \begin{lem}\label{lemProjectwodimsimple}
\begin{enumerate}
  \item $V_{i,j}\otimes \K_{\chi^{k}}\cong V_{i+k,j+3k}\cong \K_{\chi^{k}}\otimes V_{i,j}  $ and $\K_{\chi^{l}}\otimes\K_{\chi^{k}}\cong \K_{\chi^{k+l}}$ for all $(i,j)\in\Lambda,\,k,\,l\in \I_{0,5}$.

  \item $\Pp(V_{i,j})\cong V_{i,j}$ for all $(i,j)\in\Lambda$.

  \item $\Pp(\K_{\chi^i})\cong\Pp(\K_{\epsilon})\otimes \K_{\chi^i}$ and $\dim\, \Pp(\K_{\chi^i})=4$ for all $i\in \I_{0,5}$.

  \item For $(i,j),(k,l)\in\Lambda,k\in\I_{0,5}$, $\hom_{\D}(V_{i,j}\otimes V_{k,l},\K_{\chi^m})\neq 0$,  if and only if, $3(i+k)-j-l\equiv 0\mod 6$ and $m\equiv i+k+1\mod 6$.
\end{enumerate}
\end{lem}
\begin{proof}
\begin{enumerate}
  \item It follows by a direct computation.

  \item  Suppose that $\Pp(V_{i,j})\not\cong V_{i,j}$ for some $(i,j)\in\Lambda$. Since $\D$ is unimodular, $\Soc\Pp(V_{i,j})\cong V_{i,j}$ and $\dim\Pp(V_{i,j})\geq 2 \dim\,V_{i,j}$. We claim that $\dim\Pp(V_{i-k,j-3k})\geq \dim\Pp(V_{i,j})\geq 4$ for any $k\in\I_{0,5}$. Since $\Pp(V_{i-k,j-3k})\otimes\K_{\chi^k}$ is projective and
  \begin{align*}
   \hom_{\D}(\Pp(V_{i-k,j-3k})\otimes\K_{\chi^k},V_{i-k,j-3k}\otimes\K_{\chi^k})&\cong \hom_{\D}(\Pp(V_{i-k,j-3k}),V_{i-k,j-3k}\otimes\K_{\chi^k}\otimes\K_{\chi^k}\As)\\&\cong
   \hom_{\D}(\Pp(V_{i-k,j-3k},V_{i-k,j-3k})\neq 0,
   \end{align*}
      we have $\Pp(V_{i,j})\cong\Pp(V_{i-k,j-3k}\otimes\K_{\chi^k})\subset \Pp(V_{i-k,j-3k})\otimes\K_{\chi^k}$, which implies that the claim follows.
      Let $I=\{(m,n)\in\Lambda\mid(m,n)\neq (i+k,j+3k)$ for $k\in \I_{0,5}\}$. Clearly, $|I|=24$.
      \begin{align*}
      \dim \D&=\sum_{i=0}^{5}\dim\Pp(\K_{\chi^i})+\sum_{(m.n)\in I}2\dim \Pp(V_{m,n})+ \sum_{(i,j)\in\Lambda-I}2\dim \Pp(V_{i,j})\\\quad
      &\geq\sum_{i=0}^{5}\dim\Pp(\K_{\chi^i})+4|I|+8 (|\Lambda|-|I|)>4|I|+8 (|\Lambda|-|I|)=144=\dim\D,
      \end{align*}
      a contradiction. Hence $\Pp(V_{i,j})\cong V_{i,j}$ for $(i,j)\in\Lambda$.

   \item Since $\Pp(V_{i,j})\cong V_{i,j}$ for any fixed $(i,j)\in\Lambda$, it follows that $6\dim\Pp(\K_{\epsilon})=\dim\D-|\Lambda|2\dim V_{i,j}=24$ and hence $\dim\Pp(\K_{\epsilon})=4$.

  \item Since $\D$ is quasi-triangular, by
Remark \ref{rmkDmoddual} and Lemma \ref{lemProjectwodimsimple} $(1)$, $\hom_{\D}(V_{i,j}\otimes V_{k,l},\K_{\chi^m})\cong\hom_{\D}(V_{i,j} ,\K_{\chi^m}\otimes V_{k,l}\As)\cong\hom_{\D}(V_{i,j},V_{-k-1+m,-l-3+3m})$ for $m\in\I_{0,5}$. Then by Schur's lemma, $\hom_{\D}(V_{i,j}\otimes V_{k,l},\K_{\chi^m})\neq 0$, if and only if, $3(m-1)\equiv j+l\mod 6$ and  $m-1\equiv i+k\mod 6$, if and only if, $3(i+k)-j-l\equiv 0\mod 6$ and  $m\equiv i+k+1\mod 6$.
\end{enumerate}
\end{proof}
%\begin{rmk}\label{rmkSocToponedim}
%By Lemma \ref{lemProjectwodimsimple}, the top and the socle of a non-simple indecomposable $\D$-module consist of direct sums of one-dimensional $\D$-modules.
%\end{rmk}
Now we describe the projective cover $\Pp(\K_{\chi^i})$ of the simple module $\K_{\chi^i}$ for $i\in\I_{0,5}$.
\begin{defi}
Let $\theta\in\K$ such that $\theta^2=\xi^2$. Denote by $\Pp$ the left $\D$-module whose matrices defining $\D$-action with respect to a given basis $\{p_i\}_{i\in\I_{1,4}}$  are of the
form
\begin{align}
\begin{split}\label{eqprojective0}
    [a]&=\left(\begin{array}{cccc}
                                    1 & 0   &  0  & 0\\
                                    0 & \xi &  0   & 0\\
                                    0 & 0   &  \xi^{5} &0 \\
                                    0 & 0 & 0 & 1
                                 \end{array}\right),\quad
   [b]=\left(\begin{array}{cccc}
                                    0 & 0   &  0  & 0\\
                                    0 & 0 &  0   & 0\\
                                    \theta & 0   &  0 &0 \\
                                    0 & 1 & 0 & 0
                                 \end{array}\right),\\
   [g]&=\left(\begin{array}{cccc}
                                    1 & 0   &  0  & 0\\
                                    0 & -1 &  0   & 0\\
                                    0 & 0   &  -1 &0 \\
                                    0 & 0 & 0 & 1
                                 \end{array}\right),\quad
   [x]=\left(\begin{array}{cccc}
                                    0 & 0   &  0  & 0\\
                                    \theta & 0 &  0   & 0\\
                                    2 & 0   &  0 &0 \\
                                    0 & 2\theta & \xi^5 & 0
                                 \end{array}\right).
\end{split}
\end{align}
\end{defi}
%\begin{rmk}
%\begin{align*}
%[b][x]+\xi^{-2}[x][b]&=0=\theta\xi^{-2}([a]^4-[g][a]),\\
%[x]^2&=0=1-[g]^2,\\
%[a][x]+\xi^{-2}[x][a]&=\left(\begin{array}{cccc}
%                                    0 & 0   &  0  & 0\\
%                                    0 & 0 &  0   & 0\\
%                                    2(\xi^4+\xi^5) & 0   &  0 &0 \\
%                                    0 & 2\theta(1+\xi^5) & 0 & 0
%                                 \end{array}\right)=\Lam^{-1}\theta\xi^{-2}([b][a]^3-[g][b]),\\
%[b][a]&=\left(\begin{array}{cccc}
%                                    0 & 0   &  0  & 0\\
%                                    0 & 0 &  0   & 0\\
%                                    \theta & 0   &  0 &0 \\
%                                    0 & \xi & 0 & 0
%                                 \end{array}\right)=\xi[a][b].\\
%[b][g]&=\left(\begin{array}{cccc}
%                                    0 & 0   &  0  & 0\\
%                                    0 & 0 &  0   & 0\\
%                                    \theta & 0   &  0 &0 \\
%                                    0 & -1 & 0 & 0
%                                 \end{array}\right)=-[g][b].
%\end{align*}
%\end{rmk}
\begin{rmk}
It is easy to show that $\Pp$ is well-defined and  $\Soc(\Pp)\cong\K_{\epsilon}\cong\Top(\Pp)$.
\end{rmk}
\begin{lem}
$\Pp$ is an indecomposable $\D$-module.
\end{lem}
\begin{proof}
Suppose that $\Pp$ is not indecomposable. Then there exist two non-trivial submodules $M$ and $N$ such that $\Pp=M\oplus N$. We claim that $p_1\notin M$ and $p_1\notin N$. If $p_1\in M$, then $p_3=\theta^{-1}b\cdot p_1\in M$, which implies that $p_2=\theta^{-1}(x\cdot p_1-2p_3)\in M$ and $p_4=\xi x\cdot p_3\in M$. Similarly, if $p_1\in N$, then $p_2,p_3,p_4\in N$. It can not happen and hence the claim follows. Therefore, there exist some $\alpha_2,\alpha_3,\alpha_4\in\K$ such that $\alpha=p_1+\alpha_2p_2+\alpha_3p_3+\alpha_4p_4\in M$. Then $p_4=\theta^{-1}\xi(xb)\cdot\alpha\in M$ and hence $p_3=\theta^{-1}(b\cdot\alpha-\alpha_2p_4)\in M$. Therefore, $p_1+\alpha_2p_2\in M$. Since $x\cdot(p_1+\alpha_2p_2)=\theta p_2+2p_3+2\theta\alpha_2p_4\in M$, we have $p_2\in M$ and hence $p_1\in M$, a contradiction. Consequently, $\Pp$ is indecomposable.
\end{proof}

\begin{lem}\label{lemprojectivecover}
$\Pp(\K_{\epsilon})\cong \Pp$ as $\D$-modules.
\end{lem}
\begin{proof}
The proof is completely analogous to that of \cite[Lemma 2.12]{GG16}. %Indeed, $\D$ as a Drinfeld
%double of $\C^{cop}$ is Frobenius and unimodular, hence $\Pp(\K_{\epsilon})$ is an injective hull $E(\K_{\chi^i})$ of $\K_{\chi^i}$ for some $i\in\I_{0,5}$ and $\Top\Pp(\K_{\epsilon})\cong\Soc\Pp(\K_{\epsilon})$, which implies that $\Pp(\K_{\epsilon})=E(\K_{\epsilon})$. Since $\Pp$ is indecomposable and $\Soc\Pp\cong\K_{\epsilon}$, $\Pp\subset\Pp(\K_{\epsilon})$. Since $\dim\Pp=\dim\Pp(\K_{\epsilon})$, the claim follows.
\end{proof}

Now we describe the module structure of $\Pp(\K_{\chi^j})$ for $j\in\I_{0,5}$  by Lemmas \ref{lemProjectwodimsimple} $(3)$ \&  \ref{lemprojectivecover}.
\begin{cor}
Let $\{p_{i,j}\}_{i\in\I_{1,4}}$ be a linear basis of $\Pp(\K_{\chi^j})$ for $j\in\I_{0,5}$ with $p_{i,0}=p_i$. Then the $\D$-module structure of $\Pp(\K_{\chi^j})$ is given by
\begin{align}
\begin{split}\label{eqprojectivei}
a\cdot p_{i,j}&=a\cdot(p_i\otimes 1)=a\cdot p_i\otimes a\cdot 1+(\xi^4+\xi^5)ba^3\cdot p_i\otimes b\cdot 1=\xi^j(a\cdot p_i)\otimes 1,\\
b\cdot p_{i,j}&=b\cdot(p_i\otimes 1)=b\cdot p_i\otimes a\cdot 1+a^4\cdot p_i\otimes b\cdot 1=\xi^j(b\cdot p_i)\otimes 1,\\
g\cdot p_{i,j}&=g\cdot(p_i\otimes 1)=g\cdot p_i\otimes g\cdot 1=(-1)^j(g\cdot p_i)\otimes 1,\\
x\cdot p_{i,j}&=x\cdot(p_i\otimes 1)=1\cdot p_i\otimes x\cdot 1+x\cdot p_i\otimes g\cdot 1=(-1)^j(x\cdot p_i)\otimes 1.
\end{split}
\end{align}
\end{cor}

%From the preceding discussion, we have the following
\begin{thm}
The indecomposable projective covers of $\D$ consist of $\Pp(\K_{\chi^i})$ and $V_{j,k}$ for $i\in\I_{0,5}$ and $(j,k)\in\Lambda$. In particular,
\begin{gather*}
{}_{\D}\D\cong (\oplus_{i=0}^{5}\Pp(\K_{\chi^i}))\oplus(\oplus_{(i,j)\in\Lambda}V_{i,j}^2).
\end{gather*}
\end{thm}
\begin{proof}
It follows by Lemma \ref{lemProjectwodimsimple}.
\end{proof}

Now we calculate the tensor decompositions of the simple and indecomposable projective $\D$-modules. We write $\Pp_i:=\Pp(\K_{\chi^i})$ for short.
\begin{lem}\label{lem11111}
\begin{enumerate}
  \item For $i,j\in\I_{0,5}$, $\Pp_i\otimes\Pp_j=\Pp_{i+j}\oplus\Pp_{i+j}\oplus \Pp_{1+i+j}\oplus\Pp_{5+i+j}$.
  \item For $(i,j)\in\Lambda, k\in\I_{0,5}$, $V_{i,j}\otimes \Pp_k=V_{i+k,j+3k}\oplus V_{i+k,j+3k}\oplus V_{i+1+k,j+3+3k}\oplus V_{i+5+k,j+3+3k}$.
\end{enumerate}
\end{lem}
\begin{proof}
By Lemma \ref{lemProjectwodimsimple}, it suffices to prove the lemma for $\Pp_0\otimes\Pp_0$ and $V_{i,j}\otimes\Pp_0$. As $\Pp_0\otimes M$ is projective for any $\D$-module
$M$ and $[\Pp_0]=2[\K_{\epsilon}]+[\K_{\chi}]+[\K_{\chi^{5}}]$ in the Grothendieck ring $G_0(\D)$, $\Pp_0\otimes \Pp_0\cong\Pp_0\otimes(\K_{\epsilon}\oplus\K_{\epsilon}\oplus\K_{\chi}\oplus\K_{\chi^5})\cong \Pp_0\oplus\Pp_0\oplus\Pp_1\oplus\Pp_5$
and $V_{i,j}\otimes\Pp_0=V_{i,j}\oplus V_{i,j}\oplus V_{i+1,j+3}\oplus V_{i+5,j+3}$.
\end{proof}

\begin{lem}\label{lem111111}
Let $V_{i,j}$ and $V_{k,l}$ be $2$-dimensional simple $\D$-modules for $(i,j),(k,l)\in\Lambda$. Then
\begin{align*}
V_{i,j}\otimes V_{k,l}\cong
\begin{cases}
               \Pp(\K_{\chi^{i+k+1}}), &3(i+k)-j-l\equiv 0 \mod 6;\\
               V_{i+k,j+l}\oplus V_{i+k+1,j+l+3}, & \text{otherwise}.
        \end{cases}
\end{align*}
\end{lem}
\begin{proof}

If $3(i+k)-j-l\equiv 0\mod 6$, then by Lemma \ref{lemProjectwodimsimple} $(4)$, $\hom_{\D}(V_{i,j}\otimes V_{k,l},\K_{\chi^{m}})\neq 0$, if and only if, $m\equiv i+k+1\mod 6$. Since $V_{i,j}\otimes V_{k,l}$ is projective, it follows that $\Pp(\K_{\chi^{i+k+1}})\subset V_{i,j}\otimes V_{k,l}$.  Since $\dim \Pp(\K_{\chi^{i+k+1}})=\dim V_{i,j}\otimes V_{k,l}=4$, it follows that $V_{i,j}\otimes V_{k,l}\cong\Pp(\K_{\chi^{i+k+1}})$.

If $3(i+k)-j-l\not\equiv 0\mod 6$, then by Lemma \ref{lemProjectwodimsimple} $(4)$, $\hom_{\D}(V_{i,j}\otimes V_{k,l},\K_{\chi^{m}})= 0$ for all $m\in\I_{0,5}$, which implies that  $V_{i,j}\otimes V_{k,l}$ can not contain one-dimensional submodules. Hence
$V_{i,j}\otimes V_{k,l}$ must be the direct sum of two $2$-dimensional simple modules. Denote by $\{v_1,v_2\}$ and $\{w_1,w_2\}$
the linear bases of $V_{i,j}$ and $V_{k,l}$. After a direct computation, the matrices defining the action $a,g$
on $V_{i,j}\otimes V_{k,l}$ with respect to the basis $\{v_i\otimes w_j\}_{i,j\in\I_{1,2}}$ are of the following form:
\begin{align*}
\begin{split}
 [a]=\left(\begin{array}{cccc}
                                   \xi^{i+k} & 0           & 0           & \Lam^{-1}\xi^{3(k+1)}\\
                                    0        & \xi^{i+k+1} & 0           & 0\\
                                    0        & 0           & \xi^{i+k+1} & 0\\
                                    0        & 0           & 0           &\xi^{i+k+2}
                                 \end{array}\right),\quad
    [g]=\left(\begin{array}{cccc}
                                  \xi^{j+l} & 0           & 0           & 0\\
                                    0        & -\xi^{j+l} & 0           & 0\\
                                    0        & 0           & -\xi^{j+l} & 0\\
                                    0        & 0           & 0           &\xi^{j+l}
                              \end{array}\right).
\end{split}
\end{align*}
From the eigenspace decomposition with respect to the action of $a$ and $g$, we get that $V_{i,j}\otimes V_{k,l}\cong V_{i+k,j+l}\oplus V_{i+k+1,j+l+3}$.
\end{proof}

We are now able to describe the projective class ring $r_p(\D)$.
\begin{thm}
For $i,j,k,l,k+l\in\I_{1,5}$, $r_p(\D)\cong\Z[y_0,\ldots,y_5]/J$, where
\begin{align*}
J=(y_0^6=1,\ y_iy_3^2=2y_iy_0+y_iy_0^2+y_i,\  y_jy_{6-j}=y_3^2,\ y_ky_l=y_{k+l}+y_{k+l}y_0).
\end{align*}
\end{thm}
\begin{proof}
Since $\D$ is quasi-triangular, by Lemmas \ref{lemProjectwodimsimple}, \ref{lem11111} and \ref{lem111111}, $r_p(\D)$ is a commutative ring generated by $[\K_{\chi}]$, $[V_{0,j}]$ for $j\in\I_{1,5}$ satisfying the
relations $[\K_{\chi}]^6=1$, $[V_{0,i}][V_{0,3}]^2=2[V_{0,i}][\K_{\chi}]+[V_{0,i}][\K_{\chi}]^2+[V_{0,i}]$, $[V_{0,j}][V_{0,6-j}]=[V_{0,3}]^2$ and $[V_{0,k}][V_{0,l}]=[V_{0,k+l}]+[V_{0,k+l}][\K_{\chi}]$
for $i,j,k,l,k+l\in\I_{1,5}$.
 Hence we are able to construct a ring epimorphism  $\phi:\Z[y_0,\ldots,y_5]/J\mapsto r_p(\C)$ given by $\psi(y_0)=[\K_{\chi}]$ and $\psi(y_i)=[V_{0,i}]$ for $i\in\I_{1,5}$.
By Diamond Lemma, $\{y_0^i,\,y_3^2y_0^i,\, y_jy_0^i,\,i\in\I_{0,5},j\in\I_{1,5}\}$ is a $\Z$-basis of $\Z[y_0,\ldots,y_5]/J$,
then we construct the map $\psi:r_p(\D)\mapsto \Z[y_0,\ldots,y_5]/J$ by
\begin{align*}
\psi([\K_{\chi^i}])=y_0^i,\quad \psi([\Pp_j])=y_3^2y_0^{j-1},\quad \psi([V_{k,l}])=y_{l-3k}y_0^{k}.
\end{align*}
It is easy to see that $\psi$ is a well-defined ring morphism such that $\psi\phi=\id$ and $\phi\psi=\id$. Consequently, $\phi$ is a ring isomorphism.
\end{proof}

\subsection{The representation type of $\D$}Let $A$ be a finite-dimensional algebra and $\{S_1,\ldots,
S_n\}$ a complete list of non-isomorphic simple $A$-modules. The \emph{Gabriel quiver} of $A$ is the quiver $ExtQ(A)$ with vertices
$1,\dots,n$ and $\dim Ext_{A}^1(S_i,\,S_j)$ arrows from the vertex $i$ to $j$. The \emph{separated quiver} $\Gamma_{A}$ of $A$ is constructed as follows: The set of vertices is $\{S_1,\ldots, S_n, S_1\As,\ldots,S_n\As\}$ and we write $\dim\,Ext_{A}^1(S_i, S_j)$ arrows from
$S_i$ to $S_j\As$.
\begin{thm}\cite[Theorem 2.6]{ARS95}\label{thmART}
Let  $A$ be an Artin algebra with radical square zero. Then $A$ is of finite $($resp. tame$)$ representation type if and only if $\Gamma_{A}$ is
a  disjoint union of finite $($resp. affine$)$ Dynkin diagrams.
\end{thm}
\begin{lem}\cite[Lemma 4.5.]{I10}\label{lemART}
Let $J$ be the radical of $A$. Then $ExtQ(A)\cong ExtQ(A/J^2)$.
\end{lem}
\begin{defi}
For $l\in\I_{0,5},k\in\I_{0,2}$, denote by $M_l^{k}$ the $\D$-module whose matrices defining $\D$-action with respect to a fixed basis are of the form
\begin{align*}
    [a]_l^{k}&=\left(\begin{array}{ccc}
                                   \chi^l(a) & 0\\
                                    0  & \chi^{l+2k+1}(a)
                                 \end{array}\right),\quad
    [b]_l^{k}=\left(\begin{array}{ccc}
                                   0 & (1-\xi^{2(k+1)})\Lam\\
                                   0 & 0
                                 \end{array}\right),\\
    [g]_l^{k}&=\left(\begin{array}{ccc}
                                   \chi^l(g) & 0\\
                                   0   & \chi^{l+2k+1}(g)
                                 \end{array}\right),\quad
    [x]_l^{k}=\left(\begin{array}{ccc}
                                   0 & 2\theta\xi\chi^l(a^{2})\\
                                   0 & 0
                                 \end{array}\right).
\end{align*}
\end{defi}
\begin{rmk}\label{rmkIndecom-1}
It is easy to see that $M_l^{k}$ is an indecomposable left $\D$-module fitting into the exact sequence $0\rightarrow \K_{\chi^l}\rightarrow M^{k}_l\rightarrow \K_{\chi^{l+2k+1}}\rightarrow 0$. Moreover, $\Soc(M_l^{k})=\K_{\chi^{l}}$, $\Top(M_l^{k})=\K_{\chi^{l+2k+1}}$.
\end{rmk}
\begin{lem}\label{lemTwononsimpleindecom}
\begin{enumerate}
  \item Let $M$ be a $2$-dimensional non-simple indecomposable module containing $\K_{\chi^{l}}$ for $l\in\I_{0,5}$. Then
        $M\cong M_l^{k}$ for some $k\in\I_{0,2}$.
  \item
  \begin{align*}
         \dim Ext_{\D}^1(\K_{\chi^i},\K_{\chi^j})=
        \begin{cases}
               1, &\text{~if~} i\equiv j+2k+1 \mod 6 \;\text{for some}\;k\in\I_{0,2};\\
               0, & \text{otherwise}.
        \end{cases}
        \end{align*}
   \item $\dim\,Ext_{\D}^1(V_{i,j}, V_{k,\ell})=0$ for all $(i,j),(k,\ell)\in\Lambda$.
  \item $\dim\,Ext_{\D}^1(V_{i,j}, \K_{\chi^{\ell}})=0$ and $\dim\, Ext_{\D}^1(\K_{\chi^{\ell}},V_{i,j})=0$ for all $(i,j)\in\Lambda$ and $\ell\in \I_{0,5}$.
\end{enumerate}
\end{lem}
\begin{proof}
\begin{enumerate}
\item Let $A$ be the subalgebra of $\D$ generated by $a$ and $g$. Then $A$ is a finite-dimensional commutative algebra.  Let $M$ be any $2$-dimensional non-simple indecomposable $\D$-module containing $\K_{\Lam}$ with $\Lam=\chi^l$. Then $M\cong \K_{\Lam}\oplus\K_{\mu}$ as $A$-modules with $\mu$ some character on $\D$, that is, $M$ has a linear basis $\{m_1, m_2\}$ such that $\K\{ m_1\}\cong \K_{\Lam}$, $a\cdot m_2=\mu(a)m_2$, $g\cdot m_2=\mu(g)m_2$, and fits into an exact sequence
\begin{align*}
0\rightarrow \K_{\Lam}\rightarrow M\rightarrow \K_{\mu}\rightarrow 0.
\end{align*}
Then we must have that $b\cdot m_2= \alpha m_1,\ x\cdot m_2=\beta m_1$ for some $\alpha, \beta\in\K$.
Hence the matrices defining $\D$-action on $M$ with respect to $\{m_1, m_2\}$ are of the form
\begin{align*}
    [a]&=\left(\begin{array}{ccc}
                                   \Lam(a) & 0\\
                                    0  & \mu(a)
                                 \end{array}\right),\;
    [b]=\left(\begin{array}{ccc}
                                   0 & \alpha\\
                                   0 & 0
                                 \end{array}\right),\;
    [g]=\left(\begin{array}{ccc}
                                   \Lam(g) & 0\\
                                   0   & \mu(g)
                                 \end{array}\right),\;
    [x]=\left(\begin{array}{ccc}
                                   0 & \beta\\
                                   0 & 0
                                 \end{array}\right).
\end{align*}
We claim that $\Lam(g)+\mu(g)= 0$. Indeed, if $\Lam(g)+\mu(g)\neq 0$, then $\alpha=0=\beta$ by  the relations $gx=-xg$ and $gb=-bg$, which implies that $M\cong \K_{\Lam}\oplus\K_{\mu}$ as $\D$-modules, a contradiction. From the relation $bx+\xi^{-2}xb=\theta\xi^{-2}(a^4-ga)$, we have $\Lam(a^3)=\Lam(g)$ and $\mu(a^3)=\mu(g)$,
which implies that $\mu(a^3)=-\Lam(a^3)$. Hence $\mu(a)=\xi^{2k+1}\Lam(a)$ for $k\in\I_{0,2}$.
From the relation $ax+\xi^{-2}xa=\Lam^{-1}\theta\xi^{-2}(ba^3-gb)$,
we have $(\Lam(a)+\xi^{-2}\mu(a))\beta=\Lam^{-1}\theta\xi^{-2}(\mu(a^3)-\Lam(g))\alpha$. Since $\mu(a^3)=-\Lam(g)$, we have
$(\Lam(a)+\xi^{-2}\mu(a))\beta=2\Lam^{-1}\theta\xi^{-2}\mu(a^3)\alpha$.
Let $\alpha=\Lam(1-\xi^{2(k+1)})$ and $\beta=2\theta\xi\Lam(a^{2})$. Then $M\cong M_l^k$ for some $k\in\I_{0,2}$.
\item It follows by Remark \ref{rmkIndecom-1} and Lemma \ref{lemTwononsimpleindecom} $(1)$.
\item- $(4)$ It follows by Lemma \ref{lemProjectwodimsimple} $(2)$.
\end{enumerate}
\end{proof}

\begin{thm}\cite[Corollary 4.12.]{HX16}
$\D$ is of wild representation type.
\end{thm}
\begin{proof}
By  Lemma \ref{lemTwononsimpleindecom}, the Gabriel quiver $ExtQ(\D)$ of $\D$ consists of the isolated points representing $V_{i,j}$ for $(i,j)\in\Lambda$ and the quiver
$$
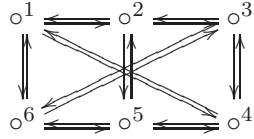
\xymatrix{
{\circ}^1\ar@{<-}[r]\ar@<-0.6mm>[r]\ar@{<-}[1,2]\ar@<-0.6mm>[1,2]  & {\circ}^{2}\ar@{<-}[r]\ar@<-0.6mm>[r]\ar@<-1.0mm>[d]  &  {\circ}^{3}\ar@{<-}[d]\ar@{<-}[1,-2]\ar@<-0.8mm>[1,-2]\ar@<-0.8mm>[d]\\
{\circ}^{6}\ar@{<-}[r]\ar@<-0.6mm>[u]\ar@{<-}[u]\ar@<-0.6mm>[r]  & {\circ}^{5}\ar@{<-}[r]\ar@{->}[u]\ar@<-0.6mm>[r]             & {\circ}^{4}
}
$$
where the vertex $i$ represents the one-dimensional simple module $\K_{\chi^{i}}$ for $i\in \I_{0,5}$. Then the separated quiver $\Gamma_{\D}$ of $\D$
contains the quivers as follows
$$
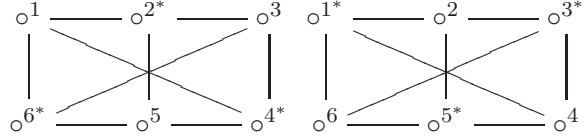
\xymatrix{
{\circ}^1\ar@{-}[r]\ar@{-}[1,2]  & {\circ}^{2\As}\ar@{-}[r]  &  {\circ}^{3}\ar@{-}[d]\ar@{-}[1,-2]\\
{\circ}^{6\As}\ar@{-}[r]\ar@{-}[u]              & {\circ}^{5}\ar@{-}[r]\ar@{-}[u]             & {\circ}^{4\As}
}
\xymatrix{
{\circ}^{1\As}\ar@{-}[r]\ar@{-}[1,2]  & {\circ}^{2}\ar@{-}[r]  &  {\circ}^{3\As}\ar@{-}[d]\ar@{-}[1,-2]\\
{\circ}^{6}\ar@{-}[r]\ar@{-}[u]              & {\circ}^{5\As}\ar@{-}[r]\ar@{-}[u]             & {\circ}^{4}
}
$$
They are not of finite type or of affine type. Therefore, the theorem follows by Theorem \ref{thmART} and Lemma \ref{lemART}.
\end{proof}

\section{Nichols algebras in $\CYD$}\label{secNicholsalg}
In this section, we determine all finite-dimensional Nichols algebras over indecomposable objects in $\CYD$ and present them by generators and relations.
\subsection{The simple and projective objects in $\CYD$}
We describe the simple and indecomposable projective objects in $\CYD$ by using the equivalence $\CYD\cong {}_{\D}\mathcal{M}$ \cite[Proposition\,10.6.16.]{M93}.

\begin{pro}\label{proYD-1}
Let $\K_{\chi^i}=\K\{v\}\in{}_{\D}\mathcal{M}$ for $ i\in \I_{0,5}$. Then $\K_{\chi^i}\in\CYD$ with the Yetter-Drinfeld module structure given by
\begin{align*}
a\cdot v=\xi^i v,\quad b\cdot v=0,\quad \delta(v)=a^{3i}\otimes v.
\end{align*}
\end{pro}
\begin{proof}
Since $\K_{\chi^i}$ is one-dimensional, the $\C$-action is given by the restriction of the character of $\D$ given in Definition $\ref{onesimple}$ and the coaction is of the form $\delta(v)=h\otimes v$ such that $\langle g, h\rangle v=(-1)^iv$, where $h\in G(\C)=\{1,a^3\}$. It follows that $a\cdot v=\xi^i v,\, b\cdot v=0$ and $\delta(v)=a^{3i}\otimes v$.
\end{proof}

\begin{pro}\label{proYD-2}
Let $V_{i,j}=\K\{v_1,v_2\}\in{}_{\D}\mathcal{M}$ for $(i,j)\in\Lambda$. Then $V_{i,j}\in\CYD$ with the Yetter-Drinfeld  module structure  given by
\begin{gather}
a\cdot v_1=\xi^iv_1,\quad b\cdot v_1=0,\quad a\cdot v_2=\xi^{i+1} v_2,\quad b\cdot v_2=v_1,\label{eqMod-Vij}\\
\delta(v_1)=a^{-j}\otimes v_1+\xi^{4}(\xi^{4i}-\xi^{i+j})ba^{-1-j}\otimes v_2,\
\delta(v_2)=a^{3-j}\otimes v_2+(\xi^{2i}+\xi^{j-i})ba^{2-j}\otimes v_1.\label{eqComod-Vij}
\end{gather}
\end{pro}
\begin{proof}
The $\C$-action is given by the restriction of the $\D$-action given in Definition $\ref{twosimple}$ and the comodule structure is given by $\delta(v)=\sum_{i=1}^{12}c_i\otimes c^i\cdot v$ for $v\in V_{i,j}$, where $\{c_i\}_{1\leq i\leq 12}$ and $\{c^i\}_{1\leq i\leq 12}$ are the dual bases of $\C$ and $\C\As$.  By Remark $\ref{rmkDDDDD}$, we have $(g^i)\As=\frac{1}{6}\sum_{j=0}^{5}\xi^{ij}a^j$ and $(g^ix)\As=\frac{1}{6\theta}\sum_{j=0}^{5}\xi^{i(j+1)}ba^j$. Then the proposition follows by a direct computation.
%\begin{align*}
%\delta(v_1)&=\sum_{k=0}^{5}(g^k)\As\otimes g^k\cdot v_1+\sum_{k=0}^{5}(g^kx)\As\otimes g^kx\cdot v_1\\
%           &=\sum_{k=0}^{5}\Lam_2^k(g^k)\As\otimes v_1+\sum_{k=0}^{5}\Lam_2^k(g^kx)\As\otimes  x_2v_2
%            =a^{-j}\otimes v_1+x_2\theta^{-1}ba^{-1-j}\otimes v_2.\\
%%\end{align*}
%%\vspace{-0.5cm}
%%\begin{align*}
%\delta(v_2)&=\sum_{k=0}^{5}(g^k)\As\otimes g^k\cdot v_2+\sum_{k=0}^{5}(g^kx)\As\otimes g^kx\cdot v_2\\
%           &=\sum_{k=0}^{5}(-\Lam_2)^k(g^k)\As\otimes v_2+\sum_{k=0}^{5}(-\Lam_2)^k(g^kx)\As\otimes  x_1v_1
%            =a^{3-j}\otimes v_2+x_1\theta^{-1}ba^{2-j}\otimes v_1.
%\end{align*}
%where $x_1=\theta^{-1}\xi^{2}\Lam_1^{-1}(\Lam_1^3+\Lam_2)$ and $x_2=\theta\xi^{4}\Lam_1(\Lam_1^3-\Lam_2)$.
%Set $\Lam_1=\xi^i,\Lam_2=\xi^j$ for simplicity.
\end{proof}
\begin{rmk}\label{rmkDmoddual}
Using the formula \eqref{eqDual-YD}, we have $V_{i,j}\As\cong V_{-i-1,-j-3}$ for all $(i,j)\in \Lambda$ .
\end{rmk}
\begin{pro}\label{proYD-3}
Let $\Pp(\K_{\chi^j})=\K\{p_{1,j},p_{2,j},p_{3,j},p_{4,j}\}\in{}_{\D}\mathcal{M}$ for $j\in \I_{0,5}$. Then $\Pp(\K_{\chi^j})\in\CYD $ with the module structure given by $\eqref{eqprojective0}$, $\eqref{eqprojectivei}$ and the comodule structure given by
\begin{align*}
\delta(p_{1,j})&=(a^3)^j\otimes p_{1,j}+\theta^{-1}(-1)^jba^5(a^3)^j\otimes (\theta p_{2,j}+2p_{3,j}),\\
\delta(p_{2,j})&=(a^3)^{j+1}\otimes p_{2,j}+2(-1)^jba^2(a^3)^j\otimes p_{4,j},\\
\delta(p_{3,j})&=(a^3)^{j+1}\otimes p_{3,j}+\xi^5\theta^{-1}(-1)^jba^2(a^3)^j\otimes p_{4,j},\quad
\delta(p_{4,j})=(a^3)^j\otimes p_{4,j}.
\end{align*}
\end{pro}
\begin{proof}
The proof follows the same lines of Proposition \ref{proYD-2}.
\end{proof}
Now we describe  braidings of the simple and indecomposable projective objects in $\CYD$.
\begin{pro}\label{braidingone}
Let $\K_{\chi^i}=\K\{v\}\in\CYD$ for $i\in\I_{0,5}$. Then $c(v\otimes v)=(-1)^iv\otimes v$.
\end{pro}
\begin{proof}
It follows by a direct computation using the formula \eqref{equbraidingYDcat} and Proposition \ref{proYD-1}.
\end{proof}

\begin{pro}\label{probraidsimpletwo}
The  braiding of $V_{i,j}=\K\{v_1,v_2\}\in\CYD$ for $(i,j)\in\Lambda$ is given by
  \begin{align*}
   c(\left[\begin{array}{ccc} v_1\\v_2\end{array}\right]\otimes\left[\begin{array}{ccc} v_1~v_2\end{array}\right])=
   \left[\begin{array}{ccc}
   \xi^{-ij}v_1\otimes v_1    & \xi^{-j(i+1)}v_2\otimes v_1+(\xi^{-ij}+\xi^{(3-j)(i+1)})v_1\otimes v_2\\
   \xi^{i(3-j)}v_1\otimes v_2   & \xi^{(3-j)(i+1)}v_2\otimes v_2+(\xi^{4i-ij+2-j}+\xi^{i-ij+2})v_1\otimes v_1
         \end{array}\right].
  \end{align*}
\end{pro}
\begin{proof}
It follows by a direct computation using the formula \eqref{equbraidingYDcat} and Proposition \ref{proYD-2}.
\end{proof}
\begin{rmk}
The braided vector spaces $V_{i,j}$ with $(i,j)\in\Lambda$ have already appeared in \cite{Hi93} and also in \cite{AGi17}. More precisely,
the braided vector spaces $V_{i,0}$ and $V_{i,3}$  belong to the case $\mathfrak{R}_{2,1}$ in \cite{Hi93,AGi17}, and the others belong to the case $\mathfrak{R}_{1,2}$.
\end{rmk}
\begin{pro}\label{probraidingprocover}
The braiding of $\Pp(\K_{\chi^j})=\K\{p_{i,j}\}_{i\in\I_{1,4}}\in\CYD$ for $j\in\I_{0,5}$ is given by
\begin{align*}
   c(p_{1,j}\otimes\left[\begin{array}{ccc} p_{1,j}\\p_{2,j}\\p_{3,j}\\p_{4,j}\end{array}\right])&=
   \left[\begin{array}{ccc} (-1)^jp_{1,j} \\p_{2,j}\\p_{3,j}\\(-1)^jp_{4,j}\end{array}\right]\otimes p_{1,j}+
   \left[\begin{array}{ccc}p_{3,j} \\ (-1)^j\theta^{-1}\xi^{-1}p_{4,j}\\0\\0\end{array}\right]\otimes (\theta p_{2,j}{+}2p_{3,j}),\\
\end{align*}
\vspace{-1.0cm}
\begin{align*}
   c(p_{2,j}\otimes\left[\begin{array}{ccc} p_{1,j}\\p_{2,j}\\p_{3,j}\\p_{4,j}\end{array}\right])&=
   \left[\begin{array}{ccc}p_{1,j} \\(-1)^{j+1}p_{2,j}\\(-1)^{j+1}p_{3,j}\\p_{4,j}\end{array}\right]\otimes p_{2,j}+
   \left[\begin{array}{ccc}2(-1)^{j+1}\xi^{-1}\theta^{-1} p_{3,j}\\2\xi^{5}(-1)^jp_{4,j}\\0\\0\end{array}\right]\otimes p_{4,j},\\
   c(p_{3,j}\otimes\left[\begin{array}{ccc} p_{1,j}\\p_{2,j}\\p_{3,j}\\p_{4,j}\end{array}\right])&=
   \left[\begin{array}{ccc}p_{1,j} \\(-1)^{j+1}p_{2,j}\\(-1)^{j+1}p_{3,j}\\p_{4,j}\end{array}\right]\otimes p_{3,j}+
   \left[\begin{array}{ccc}(-1)^{j}\xi^{-1} p_{3,j}\\\theta^{-1}\xi p_{4,j}\\0\\0\end{array}\right]\otimes p_{4,j},\\
   c(p_{4,j}\otimes\left[\begin{array}{ccc} p_{1,j}\\p_{2,j}\\p_{3,j}\\p_{4,j}\end{array}\right])&=
   \left[\begin{array}{ccc}(-1)^jp_{1,j} \\p_{2,j}\\p_{3,j}\\(-1)^jp_{4,j}\end{array}\right]\otimes p_{4,j}.
\end{align*}
\end{pro}
\begin{proof}
It follows by a direct computation using the formula \eqref{equbraidingYDcat} and Proposition \ref{proYD-3}.
\end{proof}

\subsection{Nichols algebras over the indecomposable objects in $\CYD$}
We determine all finite-dimensional Nichols algebras over the indecomposable objects in $\CYD$.  We first study the Nichols algebras over the one-dimensional objects and their projective covers in $\CYD$.
\begin{lem}\label{lemNicholbyone}
The Nichols algebra $\BN(\K_{\chi^k})$ over $\K_{\chi^k}:=\K\{v\}$ for $k\in\I_{0,5}$ is
\begin{align*}
\BN(\K_{\chi^k})=\begin{cases}
\K[v] &\text{~if~} k\in\{0,\,2,\,4\};\\
\bigwedge \K_{\chi^k} &\text{~if~} k\in\{1,\,3,\,5\}.
\end{cases}
\end{align*}
Moreover, let $V=\oplus_{i\in I}V_i$, where $V_i\cong \K_{\chi^{k_i}}$ with $k_i\in\{1,3,5\}$ and $I$ is a finite index set.
Then $\BN(V)=\bigwedge V\cong \otimes_{i\in I}\BN(V_i)$.
\end{lem}
\begin{proof}
The first claim follows immediately by Proposition $\ref{braidingone}$. Let $V_i=\K\{v_{k_i}\}$ for $i\in I$ and $k_i\in\{1,3,5\}$, then by Proposition \ref{proYD-1}, $c(v_{k_i}\otimes v_{k_j})=(-1)^{k_ik_j}v_{k_j}\otimes v_{k_i}=-v_{k_j}\otimes v_{k_i}$, which implies that $\BN(V)=\bigwedge V$ and $c_V^2=\id_V$. The last isomorphism follows by \cite[Theorem 2.2.]{G00}.
\end{proof}
\begin{lem}
The Nichols algebra $\BN(\Pp(\K_{\chi^j}))$ for $j\in\I_{0,5}$ is infinite-dimensional.
\end{lem}
\begin{proof}
By Proposition $\ref{probraidingprocover}$, we have that $c(p_{4,j}\otimes p_{4,j})=p_{4,j}\otimes p_{4,j}$ for $j\in\{0,2,4\}$ and $c(p_{3,k}\otimes p_{3,k})=p_{3,k}\otimes p_{3,k}$ for $k\in\{1,3,5\}$. Then the lemma follows by Remark \ref{rmkN-infity}.
\end{proof}

Now we show that Nichols algebras over non-simple indecomposable objects are infinite-dimensional.
%\begin{lem}
%Let $V\in\CYD$ be a finite-dimensional module such that $\dim\,\BN(V)<\infty$. Then $\dim\,\BN(W)<\infty$ for all $W\in Soc(V)$ or $W\in Top(V)$.
%\end{lem}

\begin{pro}\label{proNicholsindecom}
Let $V$ be a finite-dimensional non-simple indecomposable object in $\CYD$. Then $\dim\BN(V)=\infty$.
\end{pro}
\begin{proof}
Assume that $\dim V=2$. Then  by Lemma $\ref{lemTwononsimpleindecom}$, $V\cong M_l^{k}$ for some $k\in \I_{0,2},\,l\in\I_{0,5}$. By Remark \ref{rmkIndecom-1}, $\Soc(M_l^{k})=\K_{\chi^{l}}$ and $\Top(M_l^{k})=\K_{\chi^{l+2k+1}}$. Then by Lemma \ref{lemNicholbyone}, $\dim \BN(\Top(V))=\infty$ or $\dim\BN(\Soc(V))=\infty$, which implies that $\dim\BN(V)=\infty$.

Assume that $\dim V>2$, we prove the claim by induction on $\dim V$. As $V_{i,j}$ is projective for all $(i,j)\in\Lambda$ by Lemma \ref{lemProjectwodimsimple},  $V_{i,j}$ can not be contained in the top or socle of any non-semisimple indecomposable objects in $\CYD$. Then by Theorem \ref{thmsimplemoduleD}, $\Soc(V)$ consists of direct sums of one-dimensional objects. Let $\overline{W}$ be a simple submodule of $\Soc(V/\Soc(V))$ and $W$ be the corresponding submodule of $V$. Then $\dim\overline{W}=1$.
If $\dim\Soc(V)=1$, then $\dim W=2$ and consequently, $\dim \BN(W)=\infty$. It follows by Remark \ref{rmkN-infity} that $\dim\BN(V)=\infty$ .
If $\dim\Soc(V)>1$, then  $\K_{\chi^l}\subset\Soc(V)$ for some $l\in\I_{0,5}$. If $W/\K_{\chi^l}$ is semisimple, then $W$ contains an two-dimensional non-simple indecomposable object, which implies $\dim\BN(W)=\infty$ and hence $\dim\BN(V)=\infty$.  If $W/\K_{\chi^l}$ is not semisimple, it must contain an indecomposable object of dimension less than $\dim V$. By induction, $\dim\BN(W/\K_{\chi^l})=\infty$ and hence $\dim\,\BN(V)=\infty$.
\end{proof}

\begin{cor}\label{corSemisimpleNi}
If $\BN(V)$ is finite-dimensional, then $V$ must be semisimple.
\end{cor}
Next, we study Nichols algebras over the two-dimensional simple objects in $\CYD$.
\begin{lem}\label{lemInfty}
Let $\Lambda\As=\{(i,j)\in\Lambda\mid ij\equiv 0\mod 6\text{~or~}(i+1)(3+j)\equiv 0\mod 6\}$. Then $\dim\BN(V_{i,j})=\infty$ for all $(i,j)\in\Lambda\As$.
\end{lem}
\begin{proof}
Let $V_{i,j}=\K\{v_1,v_2\}$ for $(i,j)\in\Lambda\As$. If $ij\equiv 0\mod 6$, then by Proposition $\ref{probraidsimpletwo}$, $c(v_1\otimes v_1)=v_1\otimes v_1$. It follows by Remark \ref{rmkN-infity} that $\dim V_{i,j}=\infty$ . If $(i+1)(3+j)\equiv 0\mod 6$, then  by  Proposition $\ref{proNicholsdual}$ and Remark \ref{rmkDmoddual}, $\dim\BN(V_{i,j})=\dim\BN(V_{-i-1,-j-3})=\infty$.
\end{proof}

Now we show that $\BN(V_{i,j})$ for $(i,j)\in\Lambda-\Lambda\As$ is finite-dimensional and present them
by generators and relations. %where $V$ is either isomorphic to  $V_{1,1}$, $V_{4,2}$,
%%$V_{3,1}$, $V_{2,2}$, $V_{1,4}$, $V_{4,5}$, $V_{2,4}$, $V_{3,5}$, $V_{4,4}$, $V_{1,5}$, $V_{4,1}$ or $V_{1,2}$.

\begin{pro}\label{proV24}
 $\BN(V_{2,j+3})$  for $j\in\{1,5\}$ is generated by $v_1,v_2$ satisfying the relations:
\begin{gather}
v_2^2+\xi v_1^2=0,\quad v_1v_2-v_2v_1=0,\quad v_1^3=0.\label{eqV24}
\end{gather}
\end{pro}
\begin{proof}
We prove the assertion for $V_{2,2}$, being the proof for $V_{2,4}$ completely analogous.  By Proposition \ref{probraidsimpletwo},  $c(v_1\otimes v_1)=\xi^2v_1\otimes v_1$. It follows by the formula \eqref{eqBriadingHopfalgebra} that
\begin{align*}
\Delta(v_1^2)=v_1^2\otimes 1+v_1\otimes v_1+c(v_1\otimes v_1)+1\otimes v_1=v_1^2\otimes 1+\xi v_1\otimes v_1+1\otimes v_1^2.
\end{align*}
Then by \eqref{eqSkew-1}, we have $\partial_1(v_1^2)=\xi v_1$ and $\partial_2(v_1^2)=0$. Similarly, we obtain that
\begin{align*}
\partial_1(v_2^2)&=\xi^5v_1,\quad   \partial_2(v_2^2)=0;\quad \partial_1(v_1v_2)=\xi^2v_2,\quad \partial_2(v_1v_2)=v_1;\quad
\partial_1(v_2v_1)=\xi^2v_2,\\ \partial_2(v_2v_1)&=v_1; \quad
\partial_1(v_1^3)=0,\quad\partial_2(v_1^3)=0;\quad
\partial_1(v_1^2v_2)=\xi^5v_1v_2+(\xi^2-1)v_2v_1.%\;\partial_2(v_1^2v_2)=v_1^2.
\end{align*}
It is easy to see that $\partial_1(r)=0=\partial_2(r)$ for any relation $r$ given in \eqref{eqV24}. Then by \eqref{Def-Nichols-IV}, the quotient
$\mathfrak{B}$ of $T(V_{2,2})$ by the relations $\eqref{eqV24}$ projects onto $\BN(V_{2,2})$.

Let $I=\K\{v_1^iv_2^j,\ i\in\I_{0,2},j\in\I_{0,1}\}$. Using the relations \eqref{eqV24}, we have $v_1I\subset I$ and $v_2I\subset I$, which implies that $I$ is a left ideal of $\mathfrak{B}$. Since  $1\in I$,  $I$ linearly  generates $\mathfrak{B}$. To prove that $\BN\cong\BN(V_{2,2})$,
it suffices to show that $\{v_1^iv_2^j,i\in\I_{0,2},j\in\I_{0,1}\}$ is linearly independent in $\BN(V_{2,2})$. By Definition \ref{defiNichols}, $\{1,v_1,v_2\}$ is linearly independent in $\BN(V_{2,2})$. Let $r_2=\alpha_{2,0}v_1^2+\alpha_{1,1}v_1v_2=0$ in $\BN(V_{2,2})$. Then $0=\partial_1(r_2)=\alpha_{2,0}\xi v_1+\alpha_{1,1}\xi^2v_2$, which implies that $\alpha_{2,0}=0=\alpha_{1,1}$.  Hence $\{v_1^2, v_1v_2\}$ is linearly independent in $\BN(V_{2,2})$. Since $\partial_1(v_1^2v_2)=-v_1v_2\neq 0$ in $\BN(V_{1,2})$, $\{v_1^2v_2\}$ is linearly independent in $\BN(V_{2,2})$. Since the elements of different degree must be linearly independent, the claim follows.
 \end{proof}

\begin{pro}\label{proV35}
$\BN(V_{3,j})$ for $j\in\{1,5\}$   is generated by $v_1,v_2$ satisfying the relations:
 \begin{gather}
 v_1^2=0, \quad v_1v_2+\xi^{-j}v_2v_1=0,\quad v_2^3=0.\label{eqV35}
 \end{gather}
\end{pro}
\begin{proof}
We prove the assertion for $V_{3,1}$, being the proof for $V_{3,5}$ completely analogous. Using the braiding of $V_{3,1}$ in Proposition \ref{probraidsimpletwo} and the formula \eqref{eqSkew-1}, a direct computation shows that
\begin{align*}
\partial_1(v_1^2)&=0,\quad\partial_2(v_1^2)=0;\quad
\partial_1(v_1v_2)=\xi^2v_2,\quad\partial_2(v_1v_2)=\xi^2v_1;\\
\partial_1(v_2v_1)&=v_2,\quad\partial_2(v_2v_1)=v_1;\quad
%\partial_1(v_2^2)=-v_1,\quad
%\partial_2(v_2^2)=\xi v_2;\\
\partial_1(v_2^3)=\xi^4(v_1v_2+\xi^{-1}v_2v_1),\quad \partial_2(v_2^3)=0.
%\partial_1(v_1v_2^2)=v_1^2+\xi^4v_2^2,\quad\partial_1(v_1v_2^2)=-v_1v_2.
\end{align*}
Then the relations $\eqref{eqV35}$ are zero in $\BN(V_{3,1})$ being annihilated by $\partial_1,\partial_2$, which implies that the quotient
$\mathfrak{B}$ of $T(V_{3,1})$ by the relations $\eqref{eqV35}$ projects onto $\BN(V_{3,1})$.
Clearly, $I=\K\{v_1^iv_2^j,i\in\I_{0,1},j\in\I_{0,2}\}$ is a left ideal and hence linearly  generates $\mathfrak{B}$.
 It remains to show that $\{v_1^iv_2^j,i\in\I_{0,1},j\in\I_{0,2}\}$ is linearly independent in $\BN(V_{3,1})$. By Remark \ref{rmkDmoddual}, $V_{3,1}\As\cong V_{2,2}$. Then by Propositions \ref{proNicholsdual} and \ref{proV24}, $\dim\BN(V_{3,1})=\dim\BN(V_{2,2})=6=|I|$, which implies that the claim follows.
\end{proof}

\begin{pro}\label{proV11}
 $\BN(V_{1,j})$ for $j\in\{1,5\}$ is generated by $v_1,v_2$ satisfying the relations:
\begin{gather}
v_1^6=0, \quad v_1^2v_2+\xi^j v_2v_1^2+(1+\xi^j)v_1v_2v_1=0, \label{eq11-1}\\
v_1^3+v_2^2v_1+v_1v_2^2+v_2v_1v_2=0,\quad v_1^2v_2+v_2v_1^2+v_1v_2v_1+v_2^3=0.\label{eq11-2}
\end{gather}
\end{pro}
\begin{proof}
 We prove the assertion for $V_{1,1}$, being the proof for $V_{1,5}$ completely analogous. Using the braiding of $V_{1,1}$ in Proposition \ref{probraidsimpletwo} and the formula \eqref{eqSkew-1}, a direct computation shows that
\begin{align*}
 \partial_1(v_1v_2v_1)&=2\xi^5v_2v_1+\xi v_1v_2,\quad  \partial_2(v_1v_2v_1)=\xi^4v_1^2;\quad
\partial_1(v_2^3)=\xi^2v_1v_2+\xi^5v_2v_1,\quad \partial_2(v_2^3)=0;\\
\partial_1(v_1^2v_2)&=(\xi^4+\xi^5)v_1v_2+(\xi^2-1)v_2v_1,\quad \partial_2(v_1^2v_2)=\xi^2v_1^2;\\
\partial_1(v_2v_1^2)&=(\xi+\xi^2)v_2v_1,\quad \partial_2(v_2v_1^2)=v_1^2;\quad
\partial_1(v_2^2v_1)=-v_1^2+\xi^4v_2^2,\quad \partial_2(v_2^2v_1)=\xi^5v_2v_1;\\
\partial_1(v_1v_2^2)&=\xi^4v_2^2+(\xi+\xi^3)v_1^2,\quad  \partial_2(v_1v_2^2)=-v_1v_2;\quad
\partial_1(v_1^3)=2\xi^5v_1^2,\quad  \partial_2(v_1^3)=0;\\
\partial_1(v_2v_1v_2)&=\xi v_1^2+2\xi v_2^2,\quad  \partial_2(v_2v_1v_2)=v_1v_2+\xi^2v_2v_1;\quad\partial_1(v_1^6)=0,\quad \partial_2(v_1^6)=0.
\end{align*}
Then the relations $\eqref{eq11-1}$ and $\eqref{eq11-2}$ are zero in $\BN(V_{1,1})$ being annihilated by $\partial_1,\partial_2$, which implies that  the quotient $\mathfrak{B}$ of $T(V_{1,1})$ by the relations $\eqref{eq11-1}$ and $\eqref{eq11-2}$ projects onto $\BN(V_{1,1})$. It is easy to show that $I=\K\{v_2^i(v_1v_2)^jv_1^k,\ i\in\I_{0,2},j\in\I_{0,1},k\in\I_{0,5}\}$ is a left ideal, then $I$ linearly  generates $\mathfrak{B}$
 since clearly $1\in I$. Indeed, it suffices to show that $v_1I,v_2I\subset I$, which can be induced by
 \begin{align*}
v_2^3&=\xi^2v_2v_1^2+\xi v_1v_2v_1,\quad v_1v_2^2=-v_1^3-v_2^2v_1-v_2v_1v_2,\\
v_1^2v_2&=\xi^4v_2v_1^2+(\xi^3+\xi^4)v_1v_2v_1,\quad
(v_1v_2)^2=2\xi (v_2v_1)^2+(\xi+\xi^2)v_2^2v_1^2.
\end{align*}

To prove that $\mathfrak{B}\simeq \BN(V_{1,1})$, it suffices to show that $\{v_2^i(v_1v_2)^jv_1^k,\ i\in\I_{0,2},j\in\I_{0,1},k\in\I_{0,5}\}$ is linearly independent in $\BN(V_{1,1})$.
For this,  a direct computation shows that
\begin{align}
\begin{split}\label{eqS-1}
\partial_2(v_1^i)&=0,\quad \partial_1(v_1^i)=(i)_{\xi^5}v_1^{i-1};\quad
\partial_2(v_2v_1^i)=v_1^i,\quad \partial_1(v_2v_1^i)=\xi^2(i)_{\xi^5}v_2v_1^{i-1};\\
\partial_2(v_2^2v_1^i)&=\xi^5v_2v_1^i,\quad\partial_1(v_2^2v_1^i)=-v_1^{i+1}+(i)_{\xi^5}\xi^4v_2^2v_1^{i-1};\\
\partial_2(v_1v_2v_1^i)&=\xi^4v_1^{i+1},\quad\partial_1(v_1v_2v_1^i)=2\xi^5v_2v_1^i+\xi(i)_{\xi^5}v_1v_2v_1^{i-1};\\
\partial_2(v_2v_1v_2v_1^i)&=\xi^2v_2v_1^{i+1}+v_1v_2v_1^i,\quad \partial_1(v_2v_1v_2v_1^i)=2\xi v_2^2v_1^i+\xi v_1^{i+2}-(i)_{\xi^5}v_2v_1v_2v_1^{i-1};\\
\partial_2(v_2^2v_1v_2v_1^i)&=v_2^2v_1^{i+1},\quad
\partial_1(v_2^2v_1v_2v_1^i)=\xi^5v_2^2v_1v_2v_1^{i-1}+\xi^5v_2v_1^{i+2}+2\xi^5v_1v_2v_1^{i+1}.
\end{split}
\end{align}
It is clear that $\{1,v_1,v_2\}$ is linearly independent in $\BN(V_{1,1})$. Let $r_2=\alpha_1v_1^2+\alpha_2v_2v_1+\alpha_3v_2^2+\alpha_4v_1v_2=0$ in $\BN(V_{1,1})$ for some $\alpha_1,\ldots,\alpha_4\in\K$. From $\partial_2(r_2)=0=\partial_1(r_2)$, we obtain that
\begin{gather*}
(\alpha_2+\alpha_4\xi^4)v_1+\alpha_3\xi^5v_2=0\quad \Rightarrow\quad  \alpha_3=0=\alpha_2+\alpha_4\xi^4;\\
(\alpha_1(1+\xi^5)-\alpha_3)v_1+(\xi^2\alpha_2+2\xi^5\alpha_4)v_2=0\quad\Rightarrow\quad \alpha_1(1+\xi^5)-\alpha_3=0=\alpha_2-2\alpha_4.
\end{gather*}
It follows that $\alpha_i=0$ for $i\in\I_{1,4}$, which implies that $\{v_1^2,v_2v_1,v_2^2,v_1v_2\}$ is linearly independent in $\BN(V_{1,1})$.
Let $r_3=\beta_{1}v_1^3+\beta_{2}v_2v_1^2+\beta_{3}v_2^2v_1+\beta_{4}v_1v_2v_1+\beta_{5}v_2v_1v_2=0$ in $\BN(V_{1,1})$ for some $\beta_1,\ldots,\beta_5\in\K$. From $\partial_2(r_3)=0=\partial_1(r_3)$, we obtain that
\begin{gather*}
(\beta_2+\beta_4\xi^4)v_1^2+(\beta_5-\beta_3)\xi^2v_2v_1+\beta_5v_1v_2=0\quad\Rightarrow\quad \beta_2+\beta_4\xi^4=0=\beta_5-\beta_3=\beta_5;\\
(2\xi^5\beta_1-\beta_3+\xi\beta_5)v_1^2+[(\xi+\xi^2)\beta_2+2\xi^5\beta_4]v_2v_1+(\xi^4\beta_3+2\xi\beta_5)v_2^2+\xi\beta_4v_1v_2=0\\\quad\Rightarrow\quad
\beta_4=0=2\beta_5-\beta_3,\ 2\xi^5\beta_1-\beta_3+\xi\beta_5=0=(\xi+\xi^2)\beta_2+2\xi^5\beta_4.
\end{gather*}
It follows that $\beta_i=0$ for $i\in\I_{1,5}$, which implies that $\{v_1^3,v_2v_1^2,v_2^2v_1,v_1v_2v_1,v_2v_1v_2\}$ is linearly independent in $\BN(V_{1,1})$.

Now we claim that $\{v_2^i(v_1v_2)^jv_1^k,i\in\I_{0,2},j\in\I_{0,1},k\in\I_{0,5},i+2j+k>3\}$ is linearly independent in $\BN(V_{1,1})$.
Let $r_n=\sum_{i+2j+k=n}\alpha_{i,j,k}v_2^i(v_1v_2)^jv_1^k=0$ in $\BN(V_{1,1})$ for $(i,j,k)\in\I_{0,2}\times\I_{0,1}\times\I_{0,5}$ and $n>3$. Now we prove the claim by induction on $n$. From the equations \eqref{eqS-1}, the terms $v_2^2v_1^{k+1}$ and $v_1v_2v_1^k$  appear only one time in $\partial_{2}(r_n)$. Indeed, they  appear only in $\partial_2(v_2^2v_1v_2v_1^k)$ and $\partial_2(v_2v_1v_2v_1^k)$, respectively. Hence $\alpha_{2,1,k}=0=\alpha_{1,1,k}$ by induction. Then the term $v_2v_1^k$ appears only in $\partial_{2}(v_2^2v_1^k)$, which implies that $\alpha_{2,0,k}=0$. Furthermore, the terms $v_1v_2v_1^{k-1}$ and $v_1^{k-1}$  appear only in $\partial_1(v_1v_2v_1^k)$ and $\partial_1(v_1^k)$, respectively, which implies that $\alpha_{0,1,k}=0=\alpha_{0,0,k}$.  Then the term $v_2v_1^{k-1}$ appears only  in $\partial_1(v_2v_1^k)$, which implies that $\alpha_{1,0,k}=0$. Since the elements of different degree are linear independent,  the claim follows.
\end{proof}

\begin{pro}
$\BN(V_{4,j+3})$ for $j\in\{1,5\}$ is generated by $v_1, v_2$ satisfying the relations
\begin{gather}
v_1^3=0,\quad \xi^{2j}v_1^2v_2+\xi^{4j}v_1v_2v_1+v_2v_1^2=0,\label{eq42-1}\\
v_2^6=0,\quad v_2^2v_1+(\xi^{5j}+\xi^{4j})v_2v_1v_2-v_1v_2^2=0.\label{eq42-2}
\end{gather}
\end{pro}

\begin{proof}
 We prove the assertion for $V_{4,2}$, being the proof for $V_{4,4}$ completely analogous. Using the braiding of $V_{4,2}$ in Proposition \ref{probraidsimpletwo} and the formula \eqref{eqSkew-1}, a direct computation shows that
\begin{align*}
\partial_1(v_1v_2v_1)&=\xi^2 v_1v_2+2\xi^5v_2v_1,\quad \partial_2(v_1v_2v_1)=\xi^2 v_1^2;\quad
\partial_1(v_2v_1^2)=-v_2v_1,\quad \partial_2(v_2v_1^2)=v_1^2;\\
\partial_1(v_1^2v_2)&=-v_1v_2+(1+\xi) v_2v_1,\quad  \partial_2(v_1^2v_2)=\xi^4v_1^2;\quad
\partial_1(v_1^3)=0,\quad  \partial_2(v_1^3)=0;\\
\partial_1(v_2^2v_1)&=-v_1^2+\xi^2v_2^2,\quad  \partial_2(v_2^2v_1)=(1+\xi^5)v_2v_1;\quad
\partial_1(v_1v_2^2)=(2\xi^4+\xi^5) v_2^2+\xi v_1^2,\\ \partial_2(v_1v_2^2)&=(\xi+\xi^2)v_1v_2;\quad
\partial_1(v_2v_1v_2)=\xi^5v_1^2-2v_2^2,\quad  \partial_2(v_2v_1v_2)=v_1v_2+\xi v_2v_1.
\end{align*}
Then the relations $r$ representing the cubic relations in \eqref{eq42-1} \eqref{eq42-2}  are zero in $\BN(V_{4,2})$ being
annihilated by $\partial_1,\partial_2$. By \eqref{eqBriadingHopfalgebra}, a tedious computation shows that $\Delta(v_2^6)=v_2^6\otimes 1+1\otimes v_2^6$. Hence the
quotient $\mathfrak{B}$ of $T(V_{4,2})$ by $\eqref{eq42-1}$ and $\eqref{eq42-2}$ projects onto $\BN(V_{4,2})$.
Let $I=\K\{v_2^i(v_1v_2)^jv_1^k,\ i\in\I_{0,5},j\in\I_{0,1},k\in\I_{0,2}\}$. By \eqref{eq42-1} \eqref{eq42-2}, it is easy to see that $v_1I,v_2I\subset I$ and hence $I$ is a left ideal. Since $1\in I$,   $I$ linearly generates $\mathfrak{B}$.
It remains to show that $\{v_2^i(v_1v_2)^jv_1^k,\  i\in\I_{0,5},j\in\I_{0,1},k\in\I_{0,2}\}$ is linearly independent in $\BN(V_{4,2})$. By Remark \ref{rmkDmoddual}, $V_{4,2}\As\cong V_{1,1}$. Then by Propositions \ref{proNicholsdual} and \ref{proV11}, $\dim\BN(V_{4,2})=\dim\BN(V_{1,1})=36=|I|$, which implies that the claim follows.
\end{proof}

\begin{pro}\label{proV41}
$\BN(V_{4,j})$ for $j\in\{1,5\}$ is generated by $v_1,v_2$ satisfying the relations
\begin{gather}
v_1^3=0,\quad v_2^3-v_1^2v_2-v_2v_1^2+v_1v_2v_1=0,\label{eq41-1}\\
v_2^2v_1+v_1v_2^2-v_2v_1v_2=0,\quad \xi^j v_1^2v_2+\xi^{5j}v_2v_1^2+v_1v_2v_1=0.\label{eq41-2}
\end{gather}
\end{pro}
\begin{proof}
 We prove the assertion for $V_{4,1}$, being the proof for $V_{4,5}$ completely analogous. Using the braiding of $V_{4,1}$ in Proposition \ref{probraidsimpletwo} and the formula \eqref{eqSkew-1}, a direct computation shows that
\begin{align*}
\partial_1(v_1^3)&=0,\quad  \partial_2(v_1^3)=0;\quad
\partial_1(v_2^3)=(1+\xi)v_1v_2-(1+\xi)v_2v_1,\quad  \partial_2(v_2^3)=0;\\
\partial_1(v_1^2v_2)&=v_1v_2-\xi v_2v_1,\   \partial_2(v_1^2v_2)=\xi^2v_1^2;\quad
\partial_1(v_2^2v_1)=(\xi+\xi^2)v_1^2+\xi^4v_2^2,\   \partial_2(v_2^2v_1)=\xi^5v_2v_1;\\
\partial_1(v_2v_1^2)&=-v_2v_1,\quad \partial_2(v_2v_1^2)=v_1^2;\quad
\partial_1(v_1v_2^2)=\xi v_2^2-(1+\xi)v_1^2,\quad  \partial_2(v_1v_2^2)=v_1v_2;\\
\partial_1(v_1v_2v_1)&=\xi^4 v_1v_2,\  \partial_2(v_1v_2v_1)=\xi v_1^2;\quad
\partial_1(v_2v_1v_2)=(\xi^2-1)v_1^2,\  \partial_2(v_2v_1v_2)=v_1v_2+\xi^5v_2v_1.
\end{align*}
Then the relations $\eqref{eq41-1}$ and $\eqref{eq41-2}$ are zero in $\BN(V_{4,1})$ being annihilated by $\partial_1,\partial_2$, which implies that the quotient
 $\mathfrak{B}$ of $T(V_{4,1})$ by relations $\eqref{eq41-1}$ and $\eqref{eq41-2}$  projects onto $\BN(V_{4,1})$. Let $I=\K\{v_2^i(v_1v_2)^jv_1^k,\ i,k\in\I_{0,2},j\in\I_{0,1}\}$. We claim that $I$ is a left ideal. Indeed, it suffices to show that $v_1 I,\ v_2I\subset I$, which can be induced by
\begin{gather*}
v_1^3=0,\quad v_2^3=v_1^2v_2+v_2v_1^2-v_1v_2v_1,\\ (v_1v_2)^2=-v_2^2v_1^2,\quad
v_1v_2^2=v_2v_1v_2-v_2^2v_1,\quad v_1^2v_2=\xi v_2v_1^2+\xi^2v_1v_2v_1.
\end{gather*}
Since $1\in I$, it follows that $I$ generates $\mathfrak{B}$. To prove that $\BN\cong\BN(V_{4,1})$,  it suffices to show that $\{v_2^i(v_1v_2)^jv_1^k, \ i,k\in\I_{0,2},j\in\I_{0,1}\}$ is linearly independent
in $\BN(V_{4,1})$. For this, we obtain that
\begin{align}
\begin{split}\label{eqV41-NNN}
\partial_2(v_1^i)&=0,\quad \partial_1(v_1^i)=(i)_{\xi^2}v_1^{i-1};\quad
\partial_2(v_2v_1^i)=v_1^i,\quad \partial_1(v_2v_1^i)=\xi^2(i)_{\xi^2}v_2v_1^{i-1};\\
\partial_2(v_2^2v_1^i)&=\xi^5v_2v_1^i,\quad\partial_1(v_2^2v_1^i)=(\xi+\xi^2)v_1^{i+1}+\xi^4(i)_{\xi^2}v_2^2v_1^{i-1};\\
\partial_2(v_1v_2v_1^i)&=\xi v_1^{i+1},\quad\partial_1(v_1v_2v_1^i)=\xi^4(i)_{\xi^2}v_1v_2v_1^{i-1};\\
\partial_2(v_2v_1v_2v_1^i)&=v_1v_2v_1^i+\xi^5v_2v_1^{i+1},\quad\partial_1(v_2v_1v_2v_1^i)=(\xi^2-1)v_1^{i+2}+(i)_{\xi^2}v_2v_1v_2v_1^{i-1};\\
\partial_2(v_2^2v_1v_2v_1^i)&=\xi^5v_2v_1v_2v_1^i-v_2^2v_1^{i+1},\quad\partial_1(v_2^2v_1v_2v_1^i)=(\xi^2-1)v_2v_1^{i+2}+\xi^2(i)_{\xi^2}v_2^2v_1v_2v_1^{i-1}.
\end{split}
\end{align}
It is clear that $\{1, v_1,v_2\}$ is linearly independent in $\BN(V_{4,1})$. Let $r_2=\alpha_1v_1^2+\alpha_2v_2v_1+\alpha_3v_2^2+\alpha_4v_1v_2=0$ in $\BN(V_{4,1})$ for some $\alpha_1,\ldots,\alpha_4\in\K$. From $\partial_2(r_2)=0=\partial_1(r_2)$, we obtain that
\begin{gather*}
(\alpha_2+\alpha_4\xi)v_1+\alpha_3\xi^5v_2=0\quad \Rightarrow\quad  \alpha_3=0=\alpha_2+\alpha_4\xi;\\
(\alpha_1\xi+\alpha_3(\xi+\xi^2))v_1+\xi^2\alpha_2v_2=0\quad\Rightarrow\quad \alpha_1\xi+\alpha_3(\xi+\xi^2)=0=\alpha_2.
\end{gather*}
Hence $\alpha_i=0$ for $i\in\I_{1,4}$ and so $\{v_1^2,v_2v_1,v_2^2,v_1v_2\}$ is linearly independent in $\BN(V_{4,1})$.
Let $r_3=\beta_{1}v_2v_1^2+\beta_{2}v_2^2v_1+\beta_{3}v_1v_2v_1+\beta_{4}v_2v_1v_2=0$ in $\BN(V_{4,1})$ for some $\beta_1,\ldots,\beta_4\in\K$. From $\partial_2(r_3)=0=\partial_1(r_3)$, we obtain that
\begin{gather*}
(\beta_1+\beta_3\xi)v_1^2+(\beta_2+\beta_4)\xi^5v_2v_1+\beta_4v_1v_2=0\Rightarrow \beta_1+\beta_3\xi=0=\beta_2+\beta_4=\beta_4;\\
(\xi+\xi^2)(\beta_2+\xi\beta_4)v_1^2-\beta_1 v_2v_1+\xi^4\beta_3v_1v_2+\xi^4\beta_2v_2^2=0\Rightarrow
\beta_2+\xi\beta_4=0=\beta_1,\ \beta_3=0=\beta_2.
\end{gather*}
Hence $\beta_i=0$ for $i\in\I_{1,4}$, which implies that $\{v_2v_1^2,v_2^2v_1,v_1v_2v_1,v_2v_1v_2\}$ is linearly independent in $\BN(V_{1,4})$. Let $r_4=\gamma_1v_2^2v_1^2+\gamma_2v_1v_2v_1^2+\gamma_3(v_2v_1)^2+\gamma_4v_2^2v_1v_2=0$ in $\BN(V_{1,4})$. From the equations \eqref{eqV41-NNN}, the terms  $v_2v_1v_2$ and $v_1v_2v_1$ appear only in $\partial_2(v_2^2v_1v_2)$ and $\partial_2((v_2v_1)^2)$, respectively. Hence $\gamma_3=0=\gamma_4$. Then the terms $v_2^2v_1$ and $v_1v_2v_1$  appear only in $\partial_1(v_2^2v_1^2)$ and $\partial_1(v_1v_2v_1^2)$, respectively, which implies that $\gamma_1=0=\gamma_2$. Hence $\{v_2^2v_1v_2,(v_2v_1)^2,v_2^2v_1^2,v_1v_2v_1^2\}$ is linearly independent in $\BN(V_{4,1})$. Observe that in \eqref{eqV41-NNN} the terms $v_2v_1v_2v_1$ and $v_1v_2v_1^2$ appear only in $\partial_2(v_2^2v_1v_2v_1)$ and $\partial_2(v_2v_1v_2v_1^2)$, respectively. Hence  $\{v_2v_1v_2v_1^2, v_2^2v_1v_2v_1\}$ is also linearly independent. Since the elements of different degree must be linearly independent, the claim follows.
\end{proof}

\begin{pro}\label{proNi14}
 $\BN(V_{1,j+3})$ for $j\in\{1,5\}$ is generated by $v_1,v_2$ satisfying the relations
\begin{gather}
v_1^3=0,\quad v_2^3+v_1^2v_2+v_2v_1^2+v_1v_2v_1=0, \label{eq12-1}\\
v_2^2v_1+v_1v_2^2+v_2v_1v_2=0,\quad \xi^{2j}v_1^2v_2+\xi^{4j}v_1v_2v_1+v_2v_1^2=0.\label{eq12-2}
\end{gather}
\end{pro}
\begin{proof}
We prove the assertion for $V_{1,2}$, being the proof for $V_{1,4}$ completely analogous. Using the braiding of $V_{1,2}$ in Proposition \ref{probraidsimpletwo} and the formula \eqref{eqSkew-1}, a direct computation shows that
\begin{align*}
\partial_1(v_1^3)&=0,\quad  \partial_2(v_1^3)=0;\quad
\partial_1(v_2^3)=(\xi^2-1)(v_1v_2+v_2v_1),\quad  \partial_2(v_2^3)=0;\\
\partial_1(v_1^2v_2)&=v_1v_2-\xi^2 v_2v_1,\ \partial_2(v_1^2v_2)=\xi^4v_1^2;\quad
\partial_1(v_2^2v_1)=(\xi+\xi^2)v_1^2+\xi^2v_2^2,\  \partial_2(v_2^2v_1)=\xi v_2v_1;\\
\partial_1(v_2v_1^2)&=v_2v_1,\quad \partial_2(v_2v_1^2)=v_1^2;\quad
\partial_1(v_1v_2^2)=\xi^5 v_2^2+(1+\xi^5)v_1^2,\quad  \partial_2(v_1v_2^2)=-v_1v_2;\\
\partial_1(v_1v_2v_1)&=\xi^5 v_1v_2,\   \partial_2(v_1v_2v_1)=\xi^2 v_1^2;\quad
\partial_1(v_2v_1v_2)=(\xi^4-1)v_1^2,\   \partial_2(v_2v_1v_2)=v_1v_2+\xi^4v_2v_1.
\end{align*}
Then relations $\eqref{eq12-1}$ and $\eqref{eq12-2}$ are zero in
$\BN(V_{1,2})$ being annihilated by $\partial_1,\partial_2$, which implies that the quotient $\mathfrak{B}$ of $T(V_{1,2})$ by $\eqref{eq12-1}$ and
$\eqref{eq12-2}$ projects onto $\BN(V_{1,2})$. Let $I=\K\{v_2^i(v_1v_2)^jv_1^k,\ i,k\in\I_{0,2},j\in\I_{0,1}\}$. By \eqref{eq12-1} \eqref{eq12-2}, it is easy to see that $v_1I,v_2I\subset I$ and hence $I$ is a left ideal. Since $1\in I$, $I$ linearly generates $\mathfrak{B}$. By Remark \ref{rmkDmoddual}, $V_{1,2}\As\cong V_{4,1}$. Then by Propositions \ref{proNicholsdual} and \ref{proV41}, $\dim\BN(V_{1,2})=\dim\BN(V_{4,1})=18=|I|$. Hence  $\BN\cong\BN(V_{1,2})$.
%For this, a direct computation shows that
%\begin{align*}
%\partial_2(v_1^i)&=0,\partial_1(v_1^i)=(i)_{\xi^4}v_1^{i-1};\quad
%\partial_2(v_2v_1^i)=v_1^i,\partial_1(v_2v_1^i)=\xi(i)_{\xi^4}v_2v_1^{i-1};\\
%\partial_2(v_2^2v_1^i)&=\xi v_2v_1^i,\partial_1(v_2^2v_1^i)=(\xi+\xi^2)v_1^{i+1}+\xi^2(i)_{\xi^4}v_2^2v_1^{i-1};\\
%\partial_2(v_1v_2v_1^i)&=\xi^2v_1^{i+1},\partial_1(v_1v_2v_1^i)=\xi^5(i)_{\xi^4}v_1v_2v_1^{i-1};\\
%\partial_2(v_2v_1v_2v_1^i)&=v_1v_2v_1^i+\xi^4v_2v_1^{i+1},\partial_1(v_2v_1v_2v_1^i)=(\xi^3+\xi^4)v_1^{i+2}+(i)_{\xi^4}v_2v_1v_2v_1^{i-1};\\
%\partial_2(v_2^2v_1v_2v_1^i)&=\xi v_2v_1v_2v_1^i+\xi^4v_2^2v_1^{i+1},
%\partial_1(v_2^2v_1v_2v_1^i)=(1+\xi)v_2v_1^{i+2}+(i)_{\xi^4}v_2^2v_1v_2v_1^{i-1}.
%\end{align*}
%Let $r=\sum_{i+2j+k=n}\alpha_{i,j,k}v_2^i(v_1v_2)^jv_1^k$ be a non-trivial homogeneous relation of degree $n$.
%It's clear that $n\neq 1,2$. We claim that $n\leq 3$. Indeed, from $\partial_1(r)=0$, we have that $\alpha_{1,1,0}=\alpha_{2,0,1}=0=\alpha_{0,1,1}=\alpha_{1,0,2}$. Similarly, $n\neq 4,5,6$ and whence $\dim I=18$ in $\BN(V_{1,2})$. Consequently, the proposition is proved.
\end{proof}

\begin{proofthma}
 By Corollary $\ref{corSemisimpleNi}$,
 $V$ must be simple. A direct computation shows that $\Lambda-\Lambda\As$ consist of $(1,1)$, $(4,2)$,
$(3,1)$, $(2,2)$, $(1,4)$, $(4,5)$, $(2,4)$, $(3,5)$, $(4,4)$, $(1,5)$, $(4,1)$ and $(1,2)$. Then the rest of the assertions follow by Propositions $\ref{proV24}-\ref{proNi14}$.
\end{proofthma}

%\begin{rmk}
%By Remark $\ref{rmkDmoddual}$ and Proposition $\ref{proNicholsdual}$, $\BN(V_{i,j})\cong \BN(V_{-i-1,-j-3})^{\ast\,bop}$.
%\end{rmk}

\begin{rmk}
The Nichols algebras of dimension $6$ in Theorem $\ref{thmA}$ have already appeared in \cite{HX16} $($also \cite{AGi17}$)$. To the best of our knowledge, the Nichols algebras of dimension greater than $6$ in Theorem \ref{thmA} constitute new examples of Nichols algebras of non-diagonal type.
\end{rmk}
\begin{rmk}\label{rmk-Basic}
The authors in \cite{AA18} gave a characterization of finite-dimensional Nichols algebras over basic Hopf algebras. In particular, as stated in \cite[Example 2.14]{AA18}, the Nichols algebras in Theorem \ref{thmA} can be recovered $($up to isomorphism$)$ in a similar way.

We describe it in brief. Since $\C\As\cong\A_1$, by \cite[Proposition\,2.2.1.]{AG99}, $\CYD\cong{}_{\A_1}^{\A_1}\mathcal{YD}$ as braided monoidal categories  via the functor $(F,\eta)$ defined by \eqref{eqVHD}. More precisely, by  Remark \ref{rmkDDDDD} and Propositions \ref{proYD-1} \& \ref{proYD-2}, $F(\K_{\chi^i})\in{}_{\A_1}^{\A_1}\mathcal{YD}$ with the Yetter-Drinfeld module structure given by
\begin{align}\label{eq-Basic-1}
g\cdot v=(-1)^iv,\quad x\cdot v=0,\quad \delta(v)=g^i\otimes v;
\end{align}
 and $F(V_{i,j})\in{}_{\A_1}^{\A_1}\mathcal{YD}$ with the Yetter-Drinfeld module structure given by
\begin{align*}
g\cdot v_1=\xi^{-j}v_1,\quad x\cdot v_1=x_2\xi^{3-j}v_2,\quad g\cdot v_2=\xi^{3-j}v_2,\quad x\cdot v_2=x_1\xi^{-j}v_1;\\
\delta(v_1)=g^i\otimes v_1,\quad \delta(v_2)=g^{i+1}\otimes v_2+\theta^{-1}(-1)^{i+1}\xi^{-i-1}g^ix\otimes v_1.
\end{align*}
where $x_1=\theta^{-1}\xi^{2-i}((-1)^i+\xi^j)$ and $x_2=\theta\xi^{4+i}((-1)^i-\xi^j)$.  By \cite[Proposition\,4.2]{GM10}, $\gr\A\cong (\A)_\sigma$ with the Hopf $2$-cocycle $\sigma$ given by  $$\sigma=\epsilon\otimes\epsilon-\zeta, \text{ where } \zeta(x^ig^j,x^kg^l)=(-1)^{jk}\delta_{2,i+k} \text{ for }  i,k\in\I_{0,1}, j,l\in\I_{0,5}.$$
Hence ${}_{\A_1}^{\A_1}\mathcal{YD}\cong{}_{\gr\A_1}^{\gr\A_1}\mathcal{YD}$ via the funtors $(G,\gamma)$ given in \cite[Theorem 2.7]{MO99}. It is easy to see that $\gr\A_1\cong\BN(V)\sharp\K[\Gamma]$, where  $\Gamma\cong\Z_6$ with the generator $g$ and $V:=\K\{x\}\in{}_{\Gamma}^{\Gamma}\mathcal{YD}$ by $g\cdot x=-x$ and $\delta(x)=g\otimes x$. It should be figured out that the simple representation of $\D(\gr\A_1)$ was studied in many papers, see for example \cite{AB04} for details and more references.

By using the methods in \cite{AA18} and \cite[Theorem 1.1]{AA18}, we conclude that $\BN(GF(V_{i,j}))$ is finite-dimensional, if and only if, $\BN(Y_{i,j})$ is finite-dimensional, where $Y_{i,j}:=\K\{v_1,x\}\in{}_{\Gamma}^{\Gamma}\mathcal{YD}$ with the Yetter-Drinfeld module structure given by
\begin{gather*}
g\cdot x=-x, \quad g\cdot v_1=\xi^{-j}v_1, \quad\delta(x)=g\otimes x,\quad \delta(v_1)=g^i\otimes v_1.
\end{gather*}
In particular, $\BN(Y_{i,j})$ is of diagonal type with the Dynkin diagram \xymatrix@C+15pt{\overset{-1 }{{\circ}}\ar
@{-}[r]^{(-)^i\xi^{-j}} & \overset{\xi^{-ij}}{{\circ}}}. By \cite[Table 1]{H09}, a direct computation shows that $\dim\BN(Y_{i,j})<\infty$, if and only if, $(i,j)\in\Lambda-\Lambda\As$.
%Similarly, $\BN(GF(\K_{\chi^k}))$ is finite-dimensional, if and only if, $\BN(Z)$ is finite-dimensional, where $Z_k:=\K\{v,x\}\in{}_{\Gamma}^{\Gamma}\mathcal{YD}$ with the structure given by
%$$
%g\cdot x=-x, \  g\cdot v=(-1)^kv, \ \delta(x)=g\otimes x,\ \delta(v)=g^k\otimes v.
%$$
%Then $\BN(Z_k)$ is a quantum plane with the Dynkin diagram \xymatrix@C+5pt{\overset{-1 }{{\circ}} & \overset{(-1)^k}{{\circ}}}, hence $\dim\BN(\K_{\chi^k})<\infty$ if and only if $k\in\{1,3,5\}$.
%Moreover, by \cite[Theorem 2.9]{AA18}, $\BN(GF(V_{i,j}))\sharp\gr\A_1\cong\BN(Y_{i,j})\sharp\K[\Gamma]$ and $\BN(GF(\K_{\chi^k}))\sharp\gr\A_1\cong\BN(Z_k)\sharp\K[\Gamma]$ for $(i,j)\in\Lambda-\Lambda\As$ and $k\in\{1,3,5\}$.

%It is easy to see that $\BN(Y_{i,j})\sharp\K[\Gamma]$ and $\BN(Z_k)\sharp\K[\Gamma]$ for $(i,j)\in\Lambda-\Lambda\As$ and $k\in\{1,3,5\}$ are pairwise non-isomorphic. Observe that $\BN(F(V_{i,j}))\sharp\A_1$ is isomorphic to the quotient of $\BN(GF(V_{i,j}))\sharp\gr\A_1$ by the relation $x^2=1-g^2$. It is clear that $\gr\BN(F(V_{i,j}))\sharp\A_1\cong\BN(GF(V_{i,j}))\sharp\gr\A_1$. Similarly, $\gr\BN(F(\K_{\chi^k}))\sharp\A_1\cong\BN(GF(\K_{\chi^k}))\sharp\gr\A_1$.

It should be figured out that a similar idea was used in a submitted version of \cite{HX17} and  in \cite{X18} to discard infinite-dimensional Nichols algebras.
\end{rmk}
\begin{rmk}\label{rmk-Basic-D}
We claim that $\bigwedge F(\K_{\chi^k})\sharp\A_1$ for $k\in\{1,3,5\}$  are pairwise non-isomorphic as Hopf algebras. By \eqref{eqSmash} \eqref{eq-Basic-1}, we have $\Pp_{1,g}(\bigwedge F(\K_{\chi})\sharp\A_1)=\K\{1-g, x,v\}$, $\Pp_{1,g}(\bigwedge F(\K_{\chi^k})\sharp\A_1)=\K\{1-g,x\}$ for $k\in\{3,5\}$, $\Pp_{1,g^3}(\bigwedge F(\K_{\chi^3})\sharp\A_1)=\K\{1-g^3,v\}$ and $\Pp_{1,g^3}(\bigwedge F(\K_{\chi^5})\sharp\A_1)=\K\{1-g^3\}$, which implies that the claim follows.
\end{rmk}

\section{Hopf algebras over $\C$}\label{secHopfalgebra}
In this section, we determine all finite-dimensional Hopf algebras over $\C$ whose diagrams are strictly graded and the corresponding infinitesimal braidings are simple objects in $\CYD$.
\subsection{Generation in degree one}We show that all finite-dimensional Hopf algebras over $\C$ are generated in degree one with respect to the standard filtration, under the assumption that the diagrams are strictly graded and the infinitesimal braidings are simple objects in $\CYD$.
\begin{lem}\label{lemGeneration}
Let $S=\oplus_{n\geq 0}S(n)$ be a finite-dimensional connected $\N$-graded Hopf algebra in $\CYD$ such that $W:=S(1)$ is a simple object in $\CYD$. Assume that $S$ is generated by $S(0)\oplus S(1)$. Then $S\cong\BN(W)$.
\end{lem}
\begin{proof}

By the assumption, there is an epimorphism $S\twoheadrightarrow \BN(W)$. Then by Theorem \ref{thmA}, $W$ is isomorphic either  to $\K_{\chi^{k}}$ for $k\in\{1,3,5\}$ or $V_{i,j}$ for $(i,j)\in\Lambda-\Lambda\As$. To prove that $S\cong\BN(W)$, it suffices to show that the defining relations of $\BN(W)$ hold in $S$. This can be done by a case by case computation.

$(1)$ Assume that $W\cong\K_{\chi^{k}}$ for $k\in\{1,3,5\}$. By Lemma \ref{lemNicholbyone}, $\BN(W)\cong\K[v]/(v^2)$. Suppose that $v^2\neq 0$ in $S$.  By Propositions  \ref{braidingone} and \ref{proYD-1},  $c(v\otimes v)=-v\otimes v$ and $\delta(v)=a^3\otimes v$. Then we have
\begin{align*}
\Delta_S(v^2)=v^2\otimes 1+v\otimes v+c(v\otimes v)+1\otimes v^2=v^2\otimes 1+1\otimes v^2,\quad \delta(v^2)=\delta(v)^2=1\otimes v^2.
\end{align*}
It follows that $X:=\K\{v^2\}\subset\Pp(S)$ and $c(v^2\otimes v^2)=1\cdot v^2\otimes v^2=v^2\otimes v^2$, which implies that $\dim\BN(X)=\infty$. Since $\dim S<\infty$, the relation $v^2=0$ must hold in $A$.

$(2)$ Assume that $W\cong V_{3,j}$ for $j\in\{1,5\}$. Set $r_1=v_1v_2+\xi^{-j}v_2v_1$. By Proposition \ref{proV35}, $\BN(W)=\K\langle v_1,v_2\mid v_1^2=0,r_1=0,v_2^3=0\rangle$. By Proposition \ref{probraidsimpletwo}, we have $c(v_1\otimes v_1)=-v_1\otimes v_1$. Hence $v_1^2\in\Pp(S)$ and $c(v_1^2\otimes v_1^2)=v_1^2\otimes v_1^2$, which implies that $v_1^2=0$ in $S$.  Now we claim that $r_1=0$ in $S$. By Propositions \ref{probraidsimpletwo} and \ref{proYD-2}, a direct computation shows that $r_1\in\Pp(S)$ and
\begin{align*}
\delta(r_1)=a^{3-2j}\otimes r_1+(1-\xi^j)(\xi^j+\xi^{-j})ba^{2-2j}\otimes v_1^2=a^{3-2j}\otimes r_1,\quad a\cdot r_1=\xi r_1,\quad b\cdot r_1=0.
\end{align*}
Suppose that $r_1\neq 0$ in $S$. Set $Y:=\K\{v_1,r_1\}\subset\Pp(S)$. By \eqref{equbraidingYDcat} and Proposition \ref{probraidsimpletwo}, we have
\begin{align*}
c(v_1\otimes v_1)=-v_1\otimes v_1,\  c(v_1\otimes r_1)=\xi^{-j}r_1\otimes v_1,\quad c(r_1\otimes v_1)=-v_1\otimes r_1,\  c(r_1\otimes r_1)=\xi^jr_1\otimes r_1.
\end{align*}
That is, $Y$ is a braided vector space of diagonal type with the generalized Dynkin diagram \xymatrix@C+20pt{\overset{-1 }{{\circ}}\ar
@{-}[r]^{-\xi^{-j}} & \overset{\xi^j}{{\circ}}}.  It does not appear in \cite[Table 1]{H09}, that is, $Y$ has an infinite root system, which implies that $\dim\BN(Y)=\infty$ and hence $\dim S=\infty$, a contradiction. Therefore, the claim follows.  Since $v_2^3\in I(W)$,  $\Delta_S(v_2^3)\in T(W)\otimes I(W)+I(W)\otimes T(W)$. More precisely,
\begin{align*}
\Delta_S(v_2^3)&=v_2^3\otimes 1+1\otimes v_2^3+\sum\nolimits_{i=1}^2v_i\otimes r_i+l_i\otimes v_i,
\quad l_1,l_2,r_1,r_2\in I^2(W).
\end{align*}
Observe that $I^2(W)$ is generated by $v_1^2$ and $r_1$. Then $r_i=0=l_i$ in $S$ for $i\in\I_{1,2}$ and hence $v_2^3\in\Pp(S)$. By Proposition \ref{proYD-2}, $\delta(v_2^3)=\delta(v_2)^3=1\otimes v_2^3$. Then by the formula  \eqref{equbraidingYDcat}, $c(v_2^3\otimes v_2^3)=v_2^3\otimes v_2^3$, which implies that  the relation $v_2^3=0$ holds in $A$.

$(3)$ Assume that $W\cong V_{2,j+3}$ for $j\in\{1,5\}$. Set $r_1:=v_2^2+\xi v_1^2$ and $r_2:=v_1v_2-v_2v_1$. By Proposition \ref{proV24}, $\BN(W)=\K\langle v_1,v_2\mid r_1=0,r_2=0,v_1^3=0\rangle$. By Propositions \ref{probraidsimpletwo} and \ref{proYD-2}, a direct computation shows that $r_1, r_2\in\Pp(S)$ and
\begin{gather*}
\delta(r_2)=a^{3-2j}\otimes r_2+(1-\xi^{2j})ba^{2-2j}\otimes r_1,\quad \delta(r_1)=a^{-2j}\otimes r_1+\xi^{j+4}ba^{-1-2j}\otimes r_2.\\
a\cdot r_2=\xi^5 r_2,\quad b\cdot r_2=0,\quad a\cdot r_1=r_1, \quad b\cdot r_1=r_2.
\end{gather*}
Suppose that $r_1\neq 0$ in $S$. Then $r_2\neq 0$ in $S$ and $Z:=\K\{r_2,r_1\}\in\CYD$ is isomorphic to $V_{5,2j-3}$ with the isomorphism  given by $\phi:Z\rightarrow V_{5,2j-3}:r_2\rightarrow v_1,r_1\rightarrow v_2$.  Since $(5,2j-3)\in\Lambda^{\ast}$,  by Lemma \ref{lemInfty}, $\dim\BN(Z)=\infty$ and hence $\dim S=\infty$, a contradiction. Therefore, $r_1=0=r_2$ in $S$.
%Recall that
%\begin{align*}
%\Delta(ba^{-1-2j})=ba^{-1-2j}\otimes a^j+a^{4j}\otimes ba^{-1-2j},\quad\Delta(ba^{2-2j})=ba^{2-2j}\otimes a^{4j}+a^j\otimes ba^{2-2j}.
%\end{align*}
%Then by the formula \eqref{equbraidingYDcat} and Proposition \ref{proYD-2}, we have
%\begin{align*}
%c(r_2\otimes r_2)&=a^{3-2j}\cdot r_2\otimes r_2+(1-\xi^{2j})ba^{2-2j}\cdot r_2\otimes r_1=\xi^{-j}r_2\otimes r_2,\\
%c(r_2\otimes r_1)&=a^{3-2j}\cdot r_1\otimes r_2+(1-\xi^{2j})ba^{2-2j}\cdot r_1\otimes r_1=r_1\otimes r_2+(1-\xi^{2j})r_2\otimes r_1,\\
%c(r_1\otimes r_2)&=a^{-2j}\cdot r_2\otimes r_1+\xi^{j+4}ba^{-1-2j}\cdot r_2\otimes r_2=\xi^{-4j}r_2\otimes r_1,\\
%c(r_1\otimes r_1)&=a^{-2j}\cdot r_1\otimes r_1+\xi^{j+4}ba^{-1-2j}\cdot r_1\otimes r_2=r_1\otimes r_1+\xi^{j+4}r_2\otimes r_2\\
%\end{align*}
Then by Propositions \ref{proYD-2} and \ref{probraidsimpletwo},  we have
\begin{gather*}
\Delta_S(v_1^3)=v_1^3\otimes 1+(1+\xi^{-2j}+\xi^{-4j})v_1^2\otimes v_1+v_1^3=v_1^3\otimes 1+1\otimes v_1^3,\quad\delta(v_1^3)=\delta(v_1)^3=1\otimes v_1^3.
\end{gather*}
Hence $v_1^3\in\Pp(S)$ and $c(v_1^3\otimes v_1^3)=v_1^3\otimes v_1^3$, which implies that  $v_1^3=0$  in $S$.

$(4)$ Assume that $W\cong V_{1,j}$ for $j\in\{1,5\}$. Let $r_1:=v_1^2v_2+\xi^j v_2v_1^2+(1+\xi^j)v_1v_2v_1$,
$r_2:=v_1^3+v_2^2v_1+v_1v_2^2+v_2v_1v_2$ and $r_3:=v_1^2v_2+v_2v_1^2+v_1v_2v_1+v_2^3$ for simplicity. By Proposition \ref{proV11}, $\BN(W)=\K\langle v_1,v_2\mid v_1^6=0, r_i=0,i\in\I_{1,3} \rangle$.  By Proposition \ref{proYD-2}, we have
\begin{gather*}
a\cdot r_2=\xi^5r_2,\quad b\cdot r_2=0,\quad \delta(r_2)=a^3\otimes r_2,\\
\delta(r_1)=1\otimes r_1+2\xi^2(1+\xi^j)ba^{2-3j}\otimes r_2,\quad\delta(r_3)=1\otimes r_2+2\xi^2ba^5\otimes r_2.
\end{gather*}
Moreover, it follows by a direct computation that $r_2\in\Pp(S)$. Suppose that $r_2\neq 0$ in $S$. Set $T:=\K\{v_1,r_2\}\subset\Pp(S)$. By \eqref{equbraidingYDcat} and Proposition \ref{probraidsimpletwo}, we have
\begin{align*}
c(v_1\otimes v_1)=\xi^{-j}v_1\otimes v_1,\  c(v_1\otimes r_2)=\xi^{j}r_2\otimes v_1,\quad c(r_2\otimes v_1)=-v_1\otimes r_2,\  c(r_2\otimes r_2)=-r_2\otimes r_2.
\end{align*}
That is, $T$ is a braided vector space of diagonal type with the Dynkin diagram \xymatrix@C+15pt{\overset{-1 }{{\circ}}\ar
@{-}[r]^{-\xi^{j}} & \overset{\xi^{-j}}{{\circ}}}.  It does not appear in \cite[Table 1]{H09} and hence $\dim\BN(T)=\infty$, which implies that $\dim S=\infty$, a contradiction. It follows that $r_2=0$ in $S$ and hence $\delta(r_i)=1\otimes r_i$ for $i\in\{1,3\}$. Then by \eqref{equbraidingYDcat} and Proposition \ref{probraidsimpletwo}, a direct computation shows that  $r_i\in\Pp(S)$ and $c(r_i\otimes r_i)=r_i\otimes r_i$  for $i\in\{1,3\}$, which implies that $r_i=0$ in $S$. Finally, $\Delta_S(v_1^6)=v_1^6\otimes 1+1\otimes v_1^6$ and $\delta(v_1^6)=1\otimes v_1^6$, which implies that $v_1^6\in\Pp(S)$ and $c(v_1^6\otimes v_1^6)=v_1^6\otimes v_1^6$. Hence $v_1^6=0$ in $S$.

$(5)$ Assume that $W$ is isomorphic either to $V_{4,j}$, $V_{4,j+3}$ or $V_{1,j+3}$ for $j\in\{1,5\}$. Then by \eqref{equbraidingYDcat} and Propositions \ref{proYD-2} \& \ref{probraidsimpletwo}, a direct computation shows that $r\in\Pp(S)$ and $c(r\otimes r)=r\otimes r$ for any defining relation $r$ of $\BN(W)$,  which implies that $r=0$ in $S$. Here we only perform the case $W\cong V_{4,j}$, leaving the rest as exercise for the reader. Let  $r_1:=v_2v_1^2+\xi^jv_1v_2v_1+\xi^{2j}v_1^2v_2$, $r_2:=v_2^3-v_1^2v_2-v_2v_1^2+v_1v_2v_1$ and $r_3:=v_2^2v_1+v_1v_2^2-v_2v_1v_2$ for simplicity. By Proposition \ref{proV41}, $\BN(W)=\K\langle v_1,v_2\rangle/(v_1^3,r_1,r_2,r_3)$. By Proposition \ref{proYD-2}, we have
\begin{gather*}
\delta(v_1^3)={}a^{3}\otimes v_1^3+\xi^2(1-\xi^{j}){}ba^{-1-3j}\otimes r_1,\quad
\delta(r_1)={}1\otimes r_1,\quad
\delta(r_2)={}1\otimes r_2,\\
\delta(r_3)={}a^{3}\otimes r_3-2\xi^{4j+2}(1-\xi^{j}){}ba^{-1-3j}\otimes r_2+(\xi^{j+2}-2\xi^{2j+2}-\xi^2){}ba^{-1-3j}\otimes r_1;
\end{gather*}
Then by \eqref{equbraidingYDcat}, $c(v_1^3\otimes v_1^3)=v_1^3\otimes v_1^3$ and $c(r_i\otimes r_i)=r_i\otimes r_i$ for $i\in\I_{1,2}$. It follows by a direct computation that $v_1^3,\ r_1\in\Pp(S)$. Hence $r_1=0=v_1^3$ in $S$. Then a direct computation shows that $r_2\in\Pp(S)$ and hence $r_2=0$ in $S$. Similarly, $r_3\in\Pp(S)$ and $c(r_3\otimes r_3)=r_3\otimes r_3$, which implies that $r_3=0$ in $S$.
\end{proof}

\begin{thm}\label{thmNicholsgeneratedbyone}
Let $A$ be a finite-dimensional Hopf algebra over $\C$ such that the corresponding infinitesimal braiding $V$ is a simple object in $\CYD$. Assume that the diagram of $A$ is strictly graded. Then $A$ is generated in degree one with respect to the standard filtration.
\end{thm}
\begin{proof}
Let $R=\oplus_{n\geq 0}R(n)$ be the diagram of $A$ and $S=\oplus_{n\geq 0}S(n)$ the graded dual of $R$ in $\CYD$. It suffices to prove that $R$ is a Nichols algebra. We write $W:=S(1)$ and $V:=R(1)$. Since $W\cong V\As$ in $\CYD$ and $\dim\BN(V)<\infty$,  by  Remark \ref{rmkDmoddual} and Proposition \ref{proNicholsdual}, $W$ is a simple object in $\CYD$  such that $\dim\BN(W)<\infty$.
Since $R(1)=\Pp(R)$, by \cite[Lemma\,2.4]{AS02}, $S$ is generated by $W$ and $R$ is a Nichols algebra if and only if $\Pp(S)=W$, that is, if $S$ is a Nichols algebra. Then the theorem follows by Lemma \ref{lemGeneration}.
\end{proof}

\subsection{Liftings of Nichols algebras over $\C$}We first show that $\BN(V)\sharp\C$ does not admit non-trivial deformations, where $V$ is isomorphic either to
$\K_{\chi^{k}}$ for $k\in\{1,3,5\}$, $V_{1,1}$, $V_{4,2}$, $V_{3,1}$, $V_{2,2}$,  $V_{4,5}$, $V_{2,4}$,
 $V_{3,5}$, $V_{4,4}$, $V_{1,5}$, or $V_{4,1}$.
\begin{lem}\label{lemOnedimNichDeforma}
Let $A$ be a finite-dimensional Hopf algebra over $\C$ such that $\gr A\cong \BN(V)\sharp\C$, where  $V$ is isomorphic to
$\K_{\chi^{k}}$ for $k\in\{1,3,5\}$. Then $A\cong \bigwedge \K_{\chi^{k}}\sharp \C$.
\end{lem}
\begin{proof}
By Lemma \ref{lemNicholbyone}, $\BN(V)\cong \bigwedge V\cong \K[v]/(v^2)$. We prove the lemma by showing that the defining relations of $\gr A$ hold in $A$. By Proposition \ref{proYD-1} and the formula \eqref{eqSmash}, we have
\begin{align*}
\Delta(v^2)=\Delta(v)^2=v^2\otimes 1+(va^3+a^3v)\otimes v+1\otimes v^2=v^2\otimes 1+1\otimes v^2.
\end{align*}
Since $A$ is finite-dimensional, it follows that $\Pp(A)=0$ and hence $v^2=0$ in $A$.
\end{proof}

\begin{lem}\label{lemTwodimNichDeforma1}
Let $A$ be a finite-dimensional Hopf algebra over $\C$ such that $\gr A\cong \BN(V)\sharp\C$, where  $V$ is isomorphic either to
$V_{3,1}$ or $V_{3,5}$. Then $A\cong \BN(V)\sharp \C$.
\end{lem}
\begin{proof}
As before, it suffices to show that the defining relations in $\gr A$ hold in $A$.

Assume that $V\cong V_{3,1}$. $\BN(V_{3,1})\sharp \C$ is generated by $x, y, a, b$ satisfying the relations \eqref{eqDef1} and
%$\BN(V_{3,1})$ and $ax=-xa$, $bx=-xb$, $ay+\xi ya=\Lam^{-1}xba^3$, $by+\xi yb=xa^4$.
\begin{gather*}
 ax+xa=0= bx+xb,\quad  ay+\xi ya=\Lam^{-1}xba^3,\quad by+\xi yb=xa^4,\quad
x^2=y^3=xy-\xi^2yx=0.
\end{gather*}
%and the coalgebra structure is given by \eqref{eqDef} and
%\begin{align*}
%%\De(a)&=a\otimes a+ (\xi^4+\xi^5)b\otimes ba^3,\quad
%%\De(b)=b\otimes a^4+a\otimes b,\\
%\De(x)=x\otimes 1+a^5\otimes x+(\xi^4-\xi^2)ba^4\otimes y,\;
%\De(y)=y\otimes 1+a^2\otimes y+(1+\xi^4)ba\otimes x.
%\end{align*}
%We first calculate the following coproducts in $A$:
%\begin{align*}
%\Delta(xy-\xi^2yx)&=(xy-\xi^2yx)\otimes 1+a\otimes (xy-\xi^2yx)+(1-\xi)b\otimes x^2,\\
%\Delta(x^2)&=x^2\otimes 1+a^4\otimes x^2+(1-\xi^2)ba^3\otimes (xy-\xi^2yx),\\
%\Delta(y^3)&=y^3\otimes 1+1\otimes y^3+\xi^5\Lam^{-1}(xy-\xi^2yx)ba^4\otimes y+\Lam^{-1}\xi x^2a^2\otimes y \\&\quad +\Lam^{-1}x^2ba\otimes x+\xi^5(xy-\xi^2yx)a^5\otimes x-\xi xa\otimes (xy-\xi^2yx) \\&\quad
%-xb\otimes x^2+ba^5\otimes (-yxy+\xi y^2x+\xi^5xy^2).
%\end{align*}
%Note that $a^4$, $a$ are not group-like elements.
If the relation $x^2=0$ admits non-trivial deformations, then $x^2\in A_{[1]}$ is a linear combination
of $\{a^i,ba^i, xa^i, ya^i,xba^i,yba^i,\ i\in\I_{0,5}\}$. That is,
\begin{align*}
x^2=\sum\nolimits_{i=0}^5\alpha_i a^i+\beta_i ba^i+ \gamma_i xa^i+\lambda_i xba^i+\mu_i ya^i+\nu_i yba^i,\quad \alpha_i,\beta_i,\gamma_i,\lambda_i,\mu_i,\nu_i\in\K.
\end{align*}
Since $ax^2=x^2a$, it follows by  a direct computation that $0=ax^2-x^2a=$
\begin{align*}
\sum\nolimits_{i=0}^5\beta_i(\xi^5-1) ba^{i+1}-2\gamma_i xa^{i+1}+(\xi^5\mu_{i-2}-\lambda_i)(\xi^5+1) xba^{i+1}-\mu_i(\xi+1) ya^{i+1}
-2\nu_i yba^{i+1},
\end{align*}
which implies that $\beta_i=\gamma_i=\lambda_i=\mu_i=\nu_i=0$ for all $i\in \I_{0,5}$ and hence $x^2=\sum_{i=0}^5\alpha_i a^i$.

Since $\epsilon(x^2)=0$, we have $\sum_{i=0}^5\alpha_i=0$. Since $bx^2=x^2b$ and $ba=\xi ab$, we have $\alpha_i=0$ for all $i\in\I_{1,5}$.  Therefore, the relation $x^2=0$ must hold in $A$.
If $xy-\xi^2yx$ admits non-trivial deformations, then $xy-\xi^2yx\in A_{[1]}$, that is,
\begin{align*}
xy-\xi^2yx=\sum\nolimits_{i=0}^5\alpha_i a^i+\beta_i ba^i+ \gamma_i xa^i+\lambda_i xba^i+\mu_i ya^i+\nu_i yba^i,\quad \alpha_i,\beta_i,\gamma_i,\lambda_i,\mu_i,\nu_i\in\K.
\end{align*}
From the relations $by+\xi yb=xa^4$, $bx=-xb$ and $x^2=0$, we obtain that
\begin{align*}
a(xy-\xi^2yx)&=-\Lam^{-1}\xi x^2ba^3+\xi(xy-\xi^2yx)a=\xi(xy-\xi^2yx)a,\\
b(xy-\xi^2yx)&=-\xi x^2a^4+\xi(xy-\xi^2yx)a=\xi(xy-\xi^2yx)a.
\end{align*}
It follows that $\alpha_i=\beta_i=\gamma_i=\lambda_i=\mu_i=\nu_i=0$ for all $i\in \I_{0,5}$ and hence $xy-\xi^2yx=0$ in $A$.
Finally,  $\Delta(y^3)=\Delta(y)^3=y^3\otimes 1+1\otimes y^3$ and hence $y^3=0$  in $A$. Consequently, $A\cong \gr A$.

%If $V=V_{3,5}$,  the bosonization $\BN(V_{3,5})\sharp \C$ is generated by $x, y, a, b$ satisfying the relations \eqref{eqV31-111} and
%\begin{gather*}
%%a^6=1,\quad b^2=0,\quad ba=\xi ab,\quad ax=-xa,\quad bx=-xb,\quad ay+\xi ya=\Lam^{-1}xba^3,\quad by+\xi yb=xa^4,\\
% x^2=0,\quad y^3=0,\quad xy-\xi^4yx=0.
%\end{gather*}
%and the coalgebra structure is given by \eqref{eqDef} and
%\begin{align*}
%%\De(a)&=a\otimes a+ \Lam^{-1}b\otimes ba^3,\;
%%\De(b)=b\otimes a^4+a\otimes b,\\
%\De(x)&=x\otimes 1+a\otimes x+(\xi^4-1)b\otimes y,\;
%\De(y)=y\otimes 1+a^4\otimes y+(1+\xi^2)ba^3\otimes x.
%\end{align*}
%Then we have that
%\begin{align*}
%\Delta(xy-\xi^4yx)&=(xy-\xi^4yx)\otimes 1+a^5\otimes (xy-\xi^4yx)+\xi ba^4\otimes x^2,\\
%\Delta(x^2)&=x^2\otimes 1+a^2\otimes x^2-(1+\xi^5)ba\otimes (xy-\xi^4yx),\\
%\Delta(y^3)&=y^3\otimes 1+1\otimes y^3+\Lam^{-1}\xi^4(xy-\xi^4yx)b\otimes y+\Lam^{-1}x^2a^4\otimes y \\&\quad
%             +(1+\xi^4)(xy-\xi^4yx)a\otimes x+(1+\xi^2)\Lam^{-1}x^2ba^3\otimes x\\&\quad
%             +(1+\xi^2)ba^5\otimes(xy^2+\xi^2yxy+\xi^4y^2x)+xa^5\otimes (xy-\xi^4yx)
%             +(1+\xi^2)xba^4\otimes x^2.
%\end{align*}
%Note that $a^2$, $a^5$ are not group-like elements.
%Similarly to the above proof, since $ax^2=x^2a$ and $bx^2=x^2b$, the relation $x^2=0$ must hold in $A$. After a direct computation, we have that $a(xy-\xi^4yx)=\xi(xy-\xi^4yx)a$ and $b(xy-\xi^4yx)=\xi(xy-\xi^4yx)b$, which imply the relation $xy-\xi^4yx=0$ in $A$. Then $\Delta(y^3)=y^3\otimes 1+1\otimes y^3$ and therefore the relation $y^3=0$ must hold in $A$. Thus $A\cong gr\,A$.
The proof for $V\cong V_{3,5}$ follows the same line as for $V\cong V_{3,1}$.
\end{proof}

\begin{lem}\label{lemTwodimNichDeforma2}
Let $A$ be a finite-dimensional Hopf algebra over $\C$ such that $\gr A\cong \BN(V)\sharp\C$, where $V$ is isomorphic either to $V_{2,2}$ or $V_{2,4}$. Then $A\cong \BN(V)\sharp \C$.
\end{lem}
\begin{proof}
Assume that $V\cong V_{2,2}$. $\BN(V_{2,2})\sharp \C$ is generated by $x, y, a, b$ satisfying the relations \eqref{eqDef1} and
\begin{gather*}
ax=\xi^2xa,\quad bx=\xi^2xb,\quad ay+ya=\Lam^{-1}xba^3,\quad by+yb=xa^4,\quad
 y^2+\xi x^2=x^3= xy-yx=0.
\end{gather*}
If $xy-yx$ admits non-trivial deformations, then $xy-yx\in A_{[1]}$, that is,  $xy-yx=\sum_{i=0}^5\alpha_i a^i+\beta_i ba^i+ \gamma_i xa^i+\lambda_i xba^i+\mu_i ya^i+\nu_i yba^i$ for some $\alpha_i$, $\beta_i$, $\gamma_i$, $\lambda_i$ $\mu_i$, $\nu_i\in\K$.
Since $a(xy-yx)=\xi^5(xy-yx)a$ and $b(xy-yx)=\xi^5(xy-yx)b$,  a direct computation shows that
\begin{align*}
\alpha_i=\gamma_i=\lambda_i=\mu_i=\nu_i=0,\text{ which implies that } xy-yx=\sum\nolimits_{i=0}^5\beta_i ba^i.
\end{align*}
Moreover, we have $a(y^2+\xi x^2)=(y^2+\xi x^2)a+\Lam^{-1}(xy-yx)ba^3=(y^2+\xi x^2)a$, which implies that $y^2+\xi x^2=\sum_{i=0}^5\tau_ia^i$ for some $\tau_i\in\K$.

As $\De(x)=x\otimes 1+a^4\otimes x+(1-\xi^2)ba^3\otimes y$ and $\De(y)=y\otimes 1+a\otimes y+\xi^5b\otimes x$, we have
$\Delta(y^2+\xi x^2)=(y^2+\xi x^2)\otimes 1+a^2\otimes (y^2+\xi x^2)-ba\otimes (xy-yx)$, which implies that
\begin{gather*}
%\Delta(y^2+\xi x^2)=(y^2+\xi x^2)\otimes 1+a^2\otimes (y^2+\xi x^2)-ba\otimes (xy-yx), \text{ which implies that }\\
%\Delta(xy-yx)=(xy-yx)\otimes 1+a^5\otimes(xy-yx)+(\xi^5{-}1)(\xi^2{-}1)ba^4\otimes (y^2+\xi x^2).\label{eqxyyx}
%\Delta(x^3)=x^3\otimes 1+1\otimes x^3+(1+\xi)ba^5\otimes (\xi xyx-x^2y+\xi^5 yx^2).
\Delta(\sum\nolimits_{i=0}^5\tau_ia^i)=(\sum\nolimits_{i=0}^5\tau_ia^i)\otimes 1+a^2\otimes (\sum\nolimits_{i=0}^5\tau_ia^i)-ba\otimes (\sum\nolimits_{i=0}^5\beta_i ba^i).
\end{gather*}
It follows by a direct computation that $\beta_i=0=\tau_i$ for $i\in\I_{0,5}$. Hence $y^2+\xi x^2=0=xy-yx$ in $A$.  Finally, $\Delta(x^3)=x^3\otimes 1+1\otimes x^3$, which implies that $x^3=0$  in $A$. Consequently, $A\cong\text{gr}\,A$.

 The proof for $V\cong V_{2,4}$ follows the same line as for $V\cong V_{2,2}$.
%If $V\cong V_{2,4}$,  the bosonization $\BN(V_{2,4})\sharp \C$ is generated by $x, y, a, b$ satisfying the relations \eqref{eqV22-111}\eqref{eqV221}
%%\begin{gather*}
%%%a^6=1,\quad b^2=0,\quad ba=\xi ab,\quad ax=\xi^2xa,\quad bx=\xi^2xb,ay+ya=\Lam^{-1}xba^3,\quad by+yb=xa^4,\\
%%y^2+\xi x^2=0,\quad x^3=0,\quad xy-yx=0.
%%\end{gather*}
%and the coalgebra structure is given by \eqref{eqDef} and
%\begin{gather*}
%%\De(a)=a\otimes a+ \Lam^{-1}b\otimes ba^3,
%%\De(b)=b\otimes a^4+a\otimes b,\\
%\De(x)=x\otimes 1+a^2\otimes x+(1+\xi)ba\otimes y,
%\De(y)=y\otimes 1+a^5\otimes y+(\xi^2+\xi^4)ba^4\otimes x.
%\end{gather*}
%We first calculate the following coproducts:
%\begin{align*}
%\Delta(y^2+\xi x^2)&=(y^2+\xi x^2)\otimes 1+a^4\otimes (y^2+\xi x^2)+\xi^5ba^3\otimes (xy-yx),\\
%\Delta(xy-yx)&=(xy-yx)\otimes 1+a\otimes(xy-yx)+(1+\xi^5)b\otimes (y^2+\xi x^2).
%%\Delta(x^3)&=x^3\otimes 1+1\otimes x^3+(1+\xi)ba^5\otimes (yx^2+\xi^2 x^2y+\xi^4 xyx).
%\end{align*}
%Similarly to the above proof, after a direct calculation, we have that $a(xy-yx)=\xi^5(xy-yx)a$ and whence $xy-yx=\sum_{i=0}^5\beta_i ba^i$. From $\Delta(xy-yx)=\sum_{i=0}^5\beta_i \Delta(ba^i)$, it follows that  $y^2+\xi x^2=\alpha(1-a^4)$ and $xy-yx=-\alpha(1+\xi^5)b$. Since $a(y^2+\xi x^2)=(y^2+\xi x^2)a$, it follows that $\alpha=0$. Finally, $\Delta(x^3)=x^3\otimes 1+1\otimes x^3$. Thus $A\cong \text{gr}\,A$.
\end{proof}

\begin{lem}\label{lemTwodimNichDeforma3}
Let $A$ be a finite-dimensional Hopf algebra over $\C$ such that $\gr A\cong \BN(V)\sharp\C$, where $V$ is isomorphic either to $V_{1,1}$, $V_{1,5}$, $V_{4,2}$ or $V_{4,4}$. Then $A\cong \BN(V)\sharp \C$.
\end{lem}

\begin{proof}
Assume that $V\cong V_{1,1}$. $\BN(V_{1,1})\sharp\C$ is generated by $a,b,x,y$ satisfying the relations \ref{eqDef1} and
\begin{gather}
ax=\xi xa,\quad bx=\xi xb,\quad
ay-\xi^2ya=\Lam^{-1}xba^3,\quad by-\xi^2yb=xa^4,\quad x^6=0, \label{eq11-1-1} \\
x^2y+\xi yx^2+(1+\xi)xyx=0,\quad x^3+y^2x+xy^2+yxy=0,\quad
x^2y+yx^2+xyx+y^3=0.\label{eq11-1-2}
%\Delta(a)=a\otimes a+\Lam^{-1}b\otimes ba^3,\Delta(b)=b\otimes a^4+a\otimes b,\\
%\Delta(x)=x\otimes 1+a^5\otimes x+(\xi^2-1)ba^4\otimes y,
%\Delta(y)=y\otimes 1+a^2\otimes y+(\xi^2+1)ba\otimes x.
\end{gather}
Let $X:=x^2y+\xi yx^2+(1+\xi)xyx$, $Y:=x^3+y^2x+xy^2+yxy$ and $Z:=x^2y+yx^2+xyx+y^3$ for simplicity.
As $\Delta(x)=x\otimes 1+a^5\otimes x+(\xi^2-1)ba^4\otimes y$ and $\Delta(y)=y\otimes 1+a^2\otimes y+(\xi^2+1)ba\otimes x$,
\begin{gather*}
\Delta(X)=X\otimes 1+1\otimes X+2(\xi^2-1)ba^5\otimes Y,\\ \Delta(Y)=Y\otimes 1+a^3\otimes Y,\quad
\Delta(Z)=Z\otimes 1+1\otimes Z+2\xi^2ba^5\otimes Y.
\end{gather*}
It follows that $Y\in\Pp_{1,a^3}(A)=\Pp_{1,a^3}(\C)=\K\{1-a^3,ba^2\}$, that is, $Y=\alpha_1(1-a^3)+\alpha_2ba^2$ for some $\alpha_1,\alpha_2\in\K$. Since $aY=\xi^5Ya$, it follows that $\alpha_1=0$ and hence $Y=\alpha_2ba^5$. Therefore,
\begin{align}
\Delta(X)=X\otimes 1+1\otimes X+2(\xi^2-1)ba^5\otimes \alpha_2ba^5.\label{eqV112}
\end{align}
By \eqref{eq11-1-1} \eqref{eq11-1-2}, we have $aX=\xi^4 Xa$ and $bX=\xi^4Yb$. Similar to the proof of Lemma \ref{lemTwodimNichDeforma2},  a tedious computation on $A_{[2]}$ shows that $X=0$ in $A$ and hence the equation \eqref{eqV112} holds  only if $\alpha_2=0$.   Here we omit calculation details to save space. Therefore,  $Y=0=Z$ in $A$.
Finally,  $\Delta(x^6)=\Delta(x)^6=x^6\otimes 1+1\otimes x^6$, which implies that $x^6=0$ in $A$. Consequently, $A\cong\gr A$.

The proof for $V\cong V_{1,5}$, $V_{4,2}$ or $V_{4,4}$ follows the same line as for $V\cong V_{1,1}$.
\end{proof}

\begin{lem}
Let $A$ be a finite-dimensional Hopf algebra over $\C$ such that $\gr A\cong \BN(V)\sharp\C$, where $V$ is isomorphic either to $V_{4,1}$ or $V_{4,5}$. Then $A\cong \BN(V)\sharp \C$.
\end{lem}

\begin{proof}
Assume that $V\cong V_{4,1}$. $\BN(V_{4,1})\sharp\C$ is generated by $a,b,x,y$ satisfying the relations \eqref{eqDef1} and
\begin{gather}
ax=\xi^4xa,\quad bx=\xi^4xb,\quad
ay-\xi^5ya=(\xi^4+\xi^5)xba^3,\quad by-\xi^5yb=xa^4,\label{eqV41-1}\\
x^3=0,\quad \xi x^2y+\xi^5yx^2+xyx=0,\quad y^3-x^2y-yx^2+xyx=0,\quad y^2x+xy^2-yxy=0.\label{eqV41-2}
%\Delta(a)=a\otimes a+\Lam^{-1}b\otimes ba^3,\Delta(b)=b\otimes a^4+a\otimes b,\\
%\Delta(x)=x\otimes 1+a^5\otimes x+\xi ba^4\otimes y,
%\Delta(y)=y\otimes 1+a^2\otimes y+(\xi^2-1)ba\otimes x.
\end{gather}
Let $X:=\xi x^2y+\xi^5yx^2+xyx$, $Y:=y^3-x^2y-yx^2+xyx$ and $Z:=y^2x+xy^2-yxy$.
As $\Delta(x)=x\otimes 1+a^5\otimes x+\xi ba^4\otimes y$ and $\Delta(y)=y\otimes 1+a^2\otimes y+(\xi^2-1)ba\otimes x$, we have
\begin{align*}
\Delta(x^3)&=x^3\otimes 1+a^3\otimes x^3+\xi^2ba^2\otimes X ,\quad
\Delta(X)=X\otimes 1+1\otimes X,\\
\Delta(Y)&=Y\otimes 1+1\otimes Y,\quad
\Delta(Z)=Z\otimes 1+a^3\otimes Z+2\xi^2ba^2\otimes Y+\xi^2ba^2\otimes X.
\end{align*}
It follows that $X=0=Y$ in $A$ and hence $x^3,\; Z\in\Pp_{1,a^3}(A)=\Pp_{1,a^3}(\C)$, that is, $x^3=\alpha_1(1-a^3)+\alpha_2ba^2$ and $Z=\beta_1(1-a^3)+\beta_2ba^2$
for some $\alpha_1,\alpha_2,\beta_1,\beta_2\in\K$. Since $ax^3=x^3a$, $bx^3=x^3b$ and $aZ=\xi^2Za$, it follows that $\alpha_1=0=\alpha_2$ and $\beta_1=0=\beta_2$, which implies that  $x^3=0=Z$ in $A$. Consequently, $A\cong\gr A$.

The proof for $V\cong V_{4,5}$ follows the same line as for $V\cong V_{4,1}$.
%Assume that $V\cong V_{4,5}$, the Hopf algebra structure of $\text{gr}\,A\cong \BN(V_{4,5})\sharp\C$ is given by
%\begin{gather*}
%a^6=1,b^2=0,ba=\xi ab, ax=\xi^4 xa,bx=\xi^4 xb,
%ay-\xi^5ya=(\xi^4+\xi^5)xba^3,by-\xi^5yb=xa^4,\\
%x^3=0,y^3+xyx-x^2y-yx^2=0,
%xy^2-yxy+y^2x=0,yx^2+\xi^4x^2y+\xi^5xyx=0.\\
%\Delta(a)=a\otimes a+\Lam^{-1}b\otimes ba^3,\Delta(b)=b\otimes a^4+a\otimes b,\\
%\Delta(x)=x\otimes 1+a\otimes x-b\otimes y,\Delta(y)=y\otimes 1+a^4\otimes y+(\xi+\xi^2)ba^3\otimes x.
%\end{gather*}
%After a direct computation, we have that
%\begin{align*}
%\Delta(x^3)&=x^3\otimes 1+a^3\otimes x^3-ba^2\otimes(yx^2+\xi^5xyx+\xi^4x^2y),\\
%\Delta(yx^2+\xi^4x^2y+\xi^5xyx)&=(yx^2+\xi^4x^2y+\xi^5xyx)\otimes 1+1\otimes(yx^2+\xi^4x^2y+\xi^5xyx),\\
%\Delta(y^3+xyx-x^2y-yx^2)&=(y^3+xyx-x^2y-yx^2)\otimes 1+1\otimes(y^3+xyx-x^2y-yx^2),\\
%\Delta(xy^2-yxy+y^2x)&=(xy^2-yxy+y^2x)\otimes 1+a^3\otimes(xy^2-yxy+y^2x)\\&\quad
%2\xi^2ba^2\otimes(y^3+xyx-x^2y-yx^2)-ba^2\otimes(yx^2+\xi^4x^2y+\xi^5xyx).
%\end{align*}
%Then the relations $yx^2+\xi^4x^2y+\xi^5xyx=0,y^3+xyx-x^2y-yx^2=0$ hold in $H$ and  $x^3,xy^2-yxy+y^2x\in\Pp_{1,a^3}(H)=\K\{1-a^3,ba^2\}$, that is,
%\begin{align*}
%x^3=\alpha_1(1-a^3)+\alpha_2ba^2,xy^2-yxy+y^2x=\beta_1(1-a^3)+\beta_2ba^2,
%\end{align*}
%for some $\alpha_1,\alpha_2,\beta_1,\beta_2\in\K$.
%Since $ax^3=x^3a, bx^3=x^3b$ and $a(xy^2-yxy+y^2x)=\xi^2(xy^2-yxy+y^2x)$, it follows that $\alpha_1=0=\alpha_2$ and $\beta_1=0=\beta_2$. Hence the relations $x^3=0, xy^2-yxy+y^2x=0$ hold in $A$ and consequently $A\cong\text{gr}\,A$.
\end{proof}

Now we define two families of Hopf algebras $\mathfrak{B}_{1,2}(\mu)$ and $\mathfrak{B}_{1,4}(\mu)$ and show that they are liftings of $\BN(V_{1,2})$ and $\BN(V_{1,4})$ over $\C$.

\begin{defi}\label{defiB12}
For $\mu\in\K$ and $j\in\{1,5\}$, let $\mathfrak{B}_{1,j+3}(\mu)$ be the algebra generated by $a$, $b$, $x$, $y$ satisfying
\begin{gather*}
a^6=1,\quad b^2=0,\quad ba=\xi ab, \quad ax=\xi xa,\quad bx=\xi xb,\quad ay-\xi^2ya=\Lam^{-1}xba^3,\\ by-\xi^2yb=xa^4,\quad
x^3=0,\quad \xi^{2j} x^2y+\xi^{4j} xyx+yx^2=0,\\
y^3+x^2y+yx^2+xyx=\mu(1-a^3),\quad y^2x+xy^2+yxy=-2\xi^2\mu ba^5.
\end{gather*}
\end{defi}

$\mathfrak{B}_{1,j+3}(\mu)$ admits a Hopf algebra structure whose coalgebra structure is given by \eqref{eqDef} and
\begin{gather}
\begin{split}\label{eqDef-1214}
\Delta(x)=x\otimes 1+a^{-j-3}\otimes x+(\xi^2-\xi^{2+j})ba^{-j-4}\otimes y,\\
\Delta(y)=y\otimes 1+a^{-j}\otimes y+(\xi^2+\xi^{j+2})ba^{-j-1}\otimes x.
\end{split}
\end{gather}
%\begin{defi}\label{defiB14}
%For $\mu\in\K$, denote by $\mathfrak{B}_{1,4}(\mu)$ the Hopf algebra defined by
%\begin{gather*}
%a^6=1,b^2=0,ba=\xi ab,ax=\xi xa,bx=\xi xb,
%ay-\xi^2ya=\Lam^{-1}xba^3,by-\xi^2yb=xa^4,\\
%x^3=0,\xi^2 x^2y+yx^2+\xi^4 xyx=0.
%y^3+x^2y+yx^2+xyx=\beta_1(1-a^3),y^2x+xy^2+yxy=-2\xi^2\beta_1ba^5.\\
%\Delta(a)=a\otimes a+\Lam^{-1}b\otimes ba^3,\Delta(b)=b\otimes a^4+a\otimes b,\\
%\Delta(x)=x\otimes 1+a^2\otimes x+\xi ba\otimes y,
%\Delta(y)=y\otimes 1+a^5\otimes y+(\xi^2-1)ba^4\otimes x.
%\end{gather*}
%\end{defi}
\begin{rmk}\label{rmkB12B14}
It is easy to see that $\mathfrak{B}_{1,j+3}(0)\cong \BN(V_{1,j})\sharp\C$ and $\mathfrak{B}_{1,j+3}(0)\not\cong\mathfrak{B}_{1,j+3}(\mu)$ with $\mu\neq 0$. It is an open question to determine the isomorphism classes of $\mathfrak{B}_{1,j+3}(\mu)$.
\end{rmk}
\begin{lem}\label{lemDim}
A linear basis of $\mathfrak{B}_{1,j+3}(\mu)$ for $j\in\{1,5\}$  is given by $\{y^i(xy)^jx^kb^la^m,\ i,k\in\I_{0,2},\ j,l\in\I_{0,1},\ m\in\I_{0,5}\}$. In particular, $\dim \mathfrak{B}_{1,j+3}(\mu)=216$.
\end{lem}
\begin{proof}
We prove the lemma by the Diamond Lemma \cite{B} with the order $y<xy<x<b<a$. Let $z:=xy$ for simplicity. It suffices to show that all overlaps ambiguities are resolvable, that is, the ambiguities can be reduced to the same expression by different substitution rules.
We show that the overlaps  $\{(ay)y^2,a(y^3)\}$ and $\{(by)y^2,b(y^3)\}$ are resolvable:
\begin{align*}
(ay)y^2&=(\xi^2ya+(\xi^4+\xi^5)xba^3)y^2
=\xi^4y^2ay+(1+\xi)(xy+yx)ba^3y+(\xi^4+\xi^5)x^2ay\\
&=y^3a+(\xi^2-1)y^2xba^3+(\xi^2-1)(xy^2+yxy)ba^3+(1+\xi)(xyx+yx^2)a+(1+\xi)x^2ya\\
&=\mu(1-{}a^3)a+(\xi^2+\xi^3)(xy^2+yxy+y^2x)ba^3+\xi(xyx+yx^2+x^2y)a\\
&=\mu(1-{}a^3)a+\xi(xyx+yx^2+x^2y)a=\mu a(1-{}a^3)-ax^2y-ayx^2-axyx=a(y^3);
\end{align*}
\vspace{-0.5cm}
\begin{align*}
(by)y^2&=(\xi^2yb+xa^4)y^2
=y^3b+(1+\xi)(x^2y+xyx+yx^2)b+\xi^4(y^2x+xy^2+yxy)a^4\\
&=\mu(1-{}a^3)b+\xi(x^2y+xyx+yx^2)b+\xi^4(y^2x+xy^2+yxy)a^4\\
&=\mu b(1-{}a^3)+\xi(x^2y+xyx+yx^2)b=b(y^3).
\end{align*}
One can also show that the remaining overlaps $\{x^3x^r,x^rx^3\},\quad\{y^3y^r,y^ry^3\},\quad \{z^2z^h,z^hz^2\},$
\begin{gather*}
\{(ax)x^2, a(x^3)\},\quad \{(bx)x^2, b(x^3)\}, \quad
\{(az)z,a(z^2)\}, \quad \{(bz)z,b(z^2)\},\quad
\{(xz)z, x(z^2)\},\\ \{(x^3)z, x^2(xz)\},\quad \{(x^3)y,x^2(xy)\},\quad
\{(x(y^3),(xy)y^2\}, \quad \{z(y^3), (zy)y^2\},\quad
\{(z^2)y, z(zy)\},
\end{gather*}
are resolvable. We omit the details, since it is tedious but straightforward.
\end{proof}

\begin{pro}
$\gr\mathfrak{B}_{1,j+3}(\mu)\cong \BN(V_{1,j+3})\sharp\C$ for $j\in\{1,5\}$.
\end{pro}
\begin{proof}
Let $\mathfrak{B}_{0}$ be the Hopf subalgebra of $\mathfrak{B}_{1,j+3}(\mu)$ generated by $a,b$. We claim that $\mathfrak{B}_{0}\cong \C$. Consider the Hopf algebra map $\psi:\C\mapsto \mathfrak{B}_{1,j+3}(\mu)$ given by the composition $\C\hookrightarrow T(V_{1,j+3})\sharp \C\twoheadrightarrow \mathfrak{B}_{1,j+3}(\mu)\cong T(V_{1,j+3})\sharp \C/J^{j+3}$, where $J^{j+3}$ is the ideal generated by the last four relations in Definition \ref{defiB12}. It is clear that $\psi(\C)\cong \mathfrak{B}_{0}$ as Hopf algebras. By Lemma \ref{lemDim}, $\dim\psi(\C)=\dim\mathfrak{B}_{0}=12$. Hence  $\psi(\C)\cong \C$, which implies that the claim follows.

We claim that $(\mathfrak{B}_{1,j+3}(\mu))_{[0]}\cong\C$. Let $\mathfrak{B}_{n}=\mathfrak{B}_{n-1}+\C\{y^i(xy)^jx^k,i+2j+k=n,i,k\in\I_{0,2},j\in\I_{0,1}\}$ for $n\in\I_{1,6}$. It follows by a direct computation that $\Delta(\mathfrak{B}_{n})\subset\sum_{i=0}^n\mathfrak{B}_{n-i}\otimes\mathfrak{B}_i$ using  \eqref{eqDef} and \eqref{eqDef-1214}, that is, $\{\mathfrak{B}_{n}\}_{n\in\I_{0,6}}$ is a coalgebra filtration of $\mathfrak{B}_{1,j+3}(\mu)$. Then by \cite[Proposition 4.1.2]{R11}, the coradical $(\mathfrak{B}_{1,j+3}(\mu))_{0}\subset \mathfrak{B}_{0}\cong\C$, which implies that $(\mathfrak{B}_{1,j+3}(\mu))_{[0]}\subset \mathfrak{B}_{0}\cong\C$. By the Nichols-Zoeller Theorem, $\dim (\mathfrak{B}_{1,j+3}(\mu))_{[0]}\in\{2,3,4,6,12\}$. By Remark 2.5, $\dim (\mathfrak{B}_{1,j+3}(\mu))_{[0]}\in\{8,12\}$. Then $\dim(\mathfrak{B}_{1,j+3}(\mu))_{[0]}=12=\dim\C$ and hence $(\mathfrak{B}_{1,j+3}(\mu))_{[0]}\cong\C$.

Therefore, $\gr\mathfrak{B}_{1,j+3}(\mu)\cong R_{1,j+3}\sharp\C$, where $R_{1,j+3}$ is the diagram of $\mathfrak{B}_{1,j+3}(\mu)$. Using the formula \eqref{eQ-sh}, it is easy to see that $V_{1,j+3}\subset\Pp(R_{1,j+3})$. Since $\dim R_{1,j+3}=18=\dim\BN(V_{1,j+3})$ by Lemma \ref{lemDim},  $R_{1,j+3}\cong\BN(V_{1,j+3})$ as Hopf algebras in $\CYD$. This completes the proof.
\end{proof}

\begin{pro}\label{proB12}
Let $A$ be a finite-dimensional Hopf algebra over $\C$ such that $\gr A\cong \BN(V_{1,j+3})\sharp\C$ for $j\in\{1,5\}$. Then $A\cong \mathfrak{B}_{1,j+3}(\mu)$  for some $\mu\in\K$.
\end{pro}
\begin{proof}
Assume that $V\cong V_{1,2}$. By Remark \ref{rmkB12B14}, $\gr A\cong \mathfrak{B}_{1,2}(0)$ . Let $X:=\xi^4 x^2y+yx^2+\xi^2 xyx$, $Y:=y^3+x^2y+yx^2+xyx$ and $Z:=y^2x+xy^2+yxy$ for simplicity. By \eqref{eqDef-1214}, we obtain that
\begin{align*}
\Delta(x^3)&=x^3\otimes 1+1\otimes x^3-ba^5\otimes X,\quad
\Delta(X)=X\otimes 1+a^3\otimes X,\\
\Delta(Y)&=Y\otimes 1+a^3\otimes Y,\quad
\Delta(Z)=Z\otimes 1+1\otimes Z +2\xi^2ba^5\otimes Y+ba^5\otimes X.
\end{align*}
It follows that $X,\ Y\in\Pp_{1,a^3}(A)=\Pp_{1,a^3}(\C)$, that is, there exist $\alpha_1,\alpha_2,\beta_1,\beta_2\in\K$  such that
\begin{gather}
\notag X=\alpha_1(1-a^3)+\alpha_2ba^2,\quad Y=\beta_1(1-a^3)+\beta_2ba^2,\quad  \text{which implies that}\\
%\end{align*}
%Then we have that
%\begin{align}
\Delta(x^3-\alpha_1ba^5)=(x^3-\alpha_1ba^5)\otimes 1+1\otimes(x^3-\alpha_1ba^5)-ba^5\otimes\alpha_2ba^2.\label{eqLV12-1-1}
\end{gather}
Since $ax^3=x^3a$ and $bx^3=x^3b$, similar to the proof of Lemma \ref{lemTwodimNichDeforma2}, a tedious calculation on $A_{[2]}$ shows that $x^3=0$  in $A$ and hence \eqref{eqLV12-1-1} holds only if $\alpha_1=0=\alpha_2$. Therefore, $X=0$ in $A$.  Then
\begin{align}
\Delta(Z+2\xi^2\beta_1ba^5)&=(Z+2\xi^2\beta_1ba^5)\otimes 1+1\otimes(Z+2\xi^2\beta_1ba^5)+ 2\xi^2ba^5\otimes \beta_2ba^2.\label{eqLV12-1-2}
\end{align}
Since $aZ=\xi^5Za$,  a tedious calculation on $A_{[2]}$ shows that $Z+2\xi^2\beta_1ba^5=\sum_{i=0}^5\lambda_iba^i$ for some $\lambda_i\in\K$, which implies that
\begin{align}\label{eqLV12-1-31}
\Delta(\sum_{i=0}^5\lambda_iba^i)&=(\sum_{i=0}^5\lambda_iba^i)\otimes 1+1\otimes(\sum_{i=0}^5\lambda_iba^i)+ 2\xi^2ba^5\otimes \beta_2ba^2.
\end{align}
It follows that $\lambda_i=0$ for $i\in\I_{0,5}$ and hence \eqref{eqLV12-1-31} holds only if $\beta_2=0$. Therefore,
\begin{align*}
Y=\beta_1(1-a^3),\quad Z=-2\xi^2\beta_1ba^5.
\end{align*}
Since the defining relations of $\mathfrak{B}_{1,2}(\beta_1)$ hold in $A$,  there exists an epimorphism from $\mathfrak{B}_{1,2}(\beta_1)$ to $A$. Since $\dim A=\dim\gr A=216=\dim \mathfrak{B}_{1,2}(\beta_1)$ by Lemma \ref{lemDim},  $A\cong\mathfrak{B}_{1,2}(\beta_1)$.

The proof for $V\cong V_{1,4}$ follows the same line as for $V\cong V_{1,2}$.
\end{proof}

Finally, we are able to prove Theorem \ref{thmB}.
\begin{proofthmb}
Since $\gr A\cong R\sharp\C$ with $R(1)=V$, by Theorem \ref{thmA}, $V$ is simple and isomorphic
 to $\K_{\chi^{k}}$ for $k\in\{1,3,5\}$ or $V_{i,j}$ for $(i,j)\in\Lambda-\Lambda\As$. By Theorem $\ref{thmNicholsgeneratedbyone}$, $R\cong\BN(V)$. If $V$ is isomorphic to $\K_{\chi^i}$ for $i\in\{1,3,5\}$ or $V_{i,j}$ for $(i,j)\in\Lambda-\{(1,2),(1,4)\}-\Lambda\As$, then by Lemmas \ref{lemOnedimNichDeforma}--\ref{lemTwodimNichDeforma3}, $A\cong\BN(V)\sharp\C$. If $V\cong V_{1,j+3}$ for $j\in\{1,5\}$, then by Proposition \ref{proB12}, $A\cong\mathfrak{B}_{1,j+3}(\mu)$ for some $\mu\in\K$.

 We claim that $\BN(\K_{\chi^k})\sharp\C$ with $k\in\{1,3,5\}$ are pairwise non-isomorphic. Since  $(\BN(\K_{\chi^k})\sharp\C)\As\cong\BN(X)\sharp\A_1$ $($see for example \cite[section 2.2]{AG99}$)$, $\dim\BN(X)=2$, which implies that $\dim X=1$. Since $\CYD\cong{}_{\A_1}^{\A_1}\mathcal{YD}$, by Theorem \ref{thmA}, it follows that  $X\cong F(\K_{\chi^{n}})$ for $n\in\{1,3,5\}$. Then the claim follows by Remark \ref{rmk-Basic-D}. Since the braidings of $ \K_{\chi}$ and $V_{i,j}$ for $(i,j)\in\Lambda-\Lambda\As$ are pairwise different,
 the Hopf algebras from different families are non-isomorphic.
\end{proofthmb}

\vskip10pt \centerline{\bf ACKNOWLEDGMENT}

\vskip10pt This paper was written during the visit of the author to University of Padova supported by China Scholarship Council (No.~201706140160). The author is indebted to his  supervisors Profs. Giovanna Carnovale and Naihong Hu so much for the kind help and continued encouragement. The author is grateful to Prof. G.-A. Garcia  for helpful conversations and comments and also thanks Profs. G.-A. Garcia and L. Vendramin for providing the computation with GAP of Nichols algebras. The author would like to thank  the referee for careful reading and helpful suggestions that largely improved the exposition.

%\end{CJK}
%%%%%%%%%%%%%%%%%%%%%%%%%%%%%%%%%%%%%%%%%%%%%%%%%%%%%%%%%%%%%%%%%%%%%%


\begin{thebibliography}{50}
\bibitem[AA18]{AA18}N.~Andruskiewitsch, I.-E.~Angiono, On Nichols algebras over basic Hopf algebras, arXiv:1802.00316.
%\bibitem[A14]{A14}N. Andruskiewitsch, On finite-dimensional Hopf algebras, Proceedings of the ICM Seoul 2014 Vol. II (2014), 117--141.
%\bibitem[AAI15]{AAI15}N.~Andruskiewitsch, I.~Angiono, A.-G.~Iglesias, Liftings of Nichols algebras of diagonal type, I, Cartan type $A$, IRMN (to appear). arXiv:1509.01622v2.
%\bibitem[AAIMV14]{AAIMV14}N.~Andruskiewitsch, I.~Angiono, A.-G.~Iglesias, A.~Masuoka, C.~Vay, Lifting via cocycle deformation, J.~Pure Appl.~Algebra 218 (4) (2014), 684--703.
%\bibitem[An13]{An13}I.-E.~Angiono, On Nichols algebras of diagonal type, J. Reine Angew. Math. 2013 (683) (2011), 189--251.
%\bibitem[An15]{An15}I.-E.~Angiono, A presentation by generators and relations of Nichols algebras of diagonal type and convex orders on root systems. J.~Eur.~Math.~Soc. 17 (10) (2015), 2643--2671.
%\bibitem[AnI16]{AnI16}I.~Angiono, A.-G.~Iglesias, Liftings of Nichols algebras of diagonal type, II: All liftings are cocycle deformations, arXiv:1605.03113v1.
\bibitem[AB04]{AB04} N. Andruskiewitsch and M. Beattie, Irreducible representations of liftings of quantum planes, Lie theory and its applications in physics V (2004), 414--423.
\bibitem[AC13]{AC13}N.~Andruskiewitsch and J.~Cuadra, On the structure of (co-Frobenius) Hopf algebras, J.~ Noncomm. Geom. 7 (1) (2013), 83--104.
\bibitem[AG99]{AG99}N.~Andruskiewitsch, M.~Gra$\widetilde{n}$a, Braided Hopf algebras over non abelian finite group, Bol.~Acad.~ Nac.~Cienc.~Cordoba 63 (1999), 46--78.
\bibitem[AGi17]{AGi17}N.~Andruskiewitsch, J.-M.-J.~Giraldi, Nichols algebras that are quantum planes, 	Linear and
Multilinear Algebra, 66:5 (2018), 961--991.
%\bibitem[AGM15]{AGM15}N. Andruskiewitsch, C. Galindo, M. M\"{u}ller, Examples of finite-dimensional Hopf algebras with the dual Chevalley property, arXiv:1509.01548v2.
\bibitem[AHS10]{AHS10}N.~Andruskiewitsch, I.~Heckenberger, H.-J.~Schneider, The Nichols algebra of a semisimple Yetter-Drinfeld module, Amer. J. Math. 132 (6) (2010), 1493--1547.
%\bibitem[AN01]{AN01}N.~Andruskiewitsch, S.~Natale, Counting arguments for Hopf algebras of low dimension, Tsukuba
   Math J.~25 (1) (2001), 187--201.
\bibitem[ARS95]{ARS95}M.~Auslander, I.~Reiten and S.-O.~Smal{\o}, Representation theory of Artin algebras, Cambridge
   Studies in Adv. Math., 36. Cambridge University Press, Cambridge, 1995.
%\bibitem[AS98a]{AS}N.~Andruskiewitsch and H.-J.~Schneider, Hopf algebras of order $p^2$ and braided Hopf algebras of order $p$, J. Algebra 199 (2) (1998), 430--454.
\bibitem[AS98b]{AS98}N.~Andruskiewitsch, H.-J.~Schneider, Lifting of quantum linear spaces and pointed Hopf algebras of order $p^3$, J.~Algebra 209 (1998), 658--691.
\bibitem[AS02]{AS02}N.~Andruskiewitsch, H.-J.~Schneider, Pointed Hopf algebras, New directions in Hopf algebras, 1--68, Math.~Sci.~Res.~Inst.~Publ.~, 43, Cambridge Univ. Press, Cambridge, 2002.
%\bibitem[AS10]{AS10}N.~Andruskiewitsch, H.-J.~Schneider, On the classification of finite-dimensional pointed Hopf algebras, Annals Math., 171 (1) (2010), 375--417.
%\bibitem[B03]{B03}M. Beattie, Duals of pointed Hopf algebras, J. Algebra 262 (2003), 54--76.
\bibitem[B]{B}G. Bergman, The diamond lemma for ring theory, Adv. Math., 29 (1978), 178--218.
%\bibitem[Be03]{Be03} M. Beattie, Duals of pointed Hopf algebras, J. Algebra 262 (2003), 54--76.
\bibitem[BG13]{BG13}M.~Beattie, G.-A.~Garcia, Classifying Hopf algebras of a given dimension, Contemp. Math. 585 (2013) 125--152.
%\bibitem[F97]{F97}N.~Fukuda, Semisimple Hopf algebras of dimension $12$, Tsukuba J. Math. 21 (1) (1997), 43--54.
\bibitem[CMLS17]{CMLS17}H. Chen, H. Mohammed, W. Lin, H. Sun, The projective class rings of a family of pointed Hopf algebras of Rank two, arXiv:1709.08782.
\bibitem[G00]{G00}M.~Gra$\widetilde{n}$a, Freeness theorem for Nichols algebras, J. Algebra 231 (1) (2000), 235--257.
\bibitem[GG16]{GG16}G.-A.~Garcia, J.-M.-J.~Giraldi, On Hopf algebra over quantum subgroups, arXiv:1605.~03995, J. Pure Appl. Algebra, to appear.
%\bibitem[GV10]{GV10}G.-A.~Garcia and C.~Vay, Hopf algebras of dimension $16$, Algebr. Represent. Theory, 13 (2010),
%    383--405.
\bibitem[GM07]{GM07} L. Grunenfelder and M. Mastnak, Pointed and copointed Hopf algebras as cocycle deformations, arxiv:0709.0120.
\bibitem[GM10]{GM10} ---, Pointed Hopf algebras as cocycle deformations, arXiv:1010.4976.
%\bibitem[H06]{H06}I.~Heckenberger, The Weyl groupoid of a Nichols algebra of diagonal type, Invent. Math. 164 (1) (2006), 175--188.
\bibitem[H09]{H09}I.~Heckenberger, Classification of arithmetic root systems, Adv. Math.,  220 (1) 2009, 59--124.
\bibitem[Hi93]{Hi93} J. Hietarinta, Solving the two-dimensional constant quantum Yang-Baxter equation, J. Math. Phys. 34 (1993), 1725--1756.
%\bibitem[HN09]{HN09}M.~Hilgemann, Siu-Hung Ng, Hopf algebras of dimension $2p^2$, J.~Lond.~Math.~Soc. 80 (2) (2009), 295--310.
%\bibitem[HS10]{HS10}I.~Heckenberger and H.-J.~Schneider, Nichols algebras over groups with finite root
%system of rank two I, J. Algebra 324 (11) (2010), 3090--3114.
%\bibitem[HS101]{HS101}I.~Heckenberger and H.-J.~Schneider, Root systems and Weyl groupoids for Nichols algebras, Proc. Lond. Math. Soc. 101 (3) (2010), 623--654.
%\bibitem[HS13]{HS13}---, Yetter-Drinfeld modules over bosonizations of dually paired Hopf algebras, Advances in Mathematics, 244 (244) (2013), 354--394.
%\bibitem[HV14]{HV14}I.~Heckenberger, L.~Vendramin, Nichols algebras over groups with finite root system of rank two II, J. Group Theory 17 (6) (2014), 1009--1034.
%\bibitem[HV15]{HV15}I.~Heckenberger, L.~Vendramin, Nichols algebras over groups with finite root system of rank two III, J. Algebra 422 (2015), 223--256.
\bibitem[HX16]{HX16}N.~Hu, R. Xiong, Some Hopf algebras of dimension 72 without the Chevalley property,  arXiv:1612.04987.
\bibitem[HX17]{HX17}---, Eight classes of new Hopf algebras of dimension 128 without the Chevalley property,  arXiv:1701.01991.
%\bibitem[HY08]{HY08}I.~Heckenberger, H.~Yamane, A generalization of Coxeter groups, root systems, and Matsumoto's
%    theorem, Math. Z. 259 (2008), 255--276.
\bibitem[I10]{I10}A.-G. Iglesias, Representations of pointed Hopf algebras over $S_3$, Revista
de la Uni\'{o}n Matem\'{a}tica Argentina 51 (2010), 51--77.
\bibitem[M93]{M93}S.~Montgomery, Hopf Algebras and their Actions on Rings, CMBS Reg. Conf. Ser. in Math. 82,
    Amer.~Math.~Soc. 1993.
%\bibitem[Ma95]{Ma95}A.~Masuoka, Semisimple Hopf algebras of dimension $2p$, Comm. Algebra 23 (1995), 1931--1940.
%\bibitem[Ma96]{Ma96}A.~Masuoka, The $p^n$ theorem for semisimple Hopf algebras, Proc. Amer. Math. Soc. 124 (3) (1996), 735--737.
%\bibitem[Ma08]{Ma08}A.~Masuoka, Abelian and non-abelian second cohomologies of quantized enveloping algebras, J. Algebra 320 (1) (2008), 1--47.
\bibitem[MO99]{MO99} S. Majid and R. Oeckl, Twisting of quantum differentials and the Planck scale Hopf algebra, Comm. Math. Phys. 205 (3) (1999), 617--655.
\bibitem[N78]{N78}W.-D.~Nichols, Bialgebras of type one, Comm. Algebra 6 (15) (1978), 1521--1552.
\bibitem[Na02]{Na02}S.~Natale, Hopf algebras of dimension $12$, Algebr.~Represent.~Theory 5 (5) (2002), 445--455.
%\bibitem[Ng02]{Ng02}S.-H.~Ng, Non-semisimple Hopf algebras of dimension $p^2$, J. Algebra 255 (1) (2002), 182--197.
%\bibitem[Ng05]{Ng05}S.-H.~Ng, Hopf algebras of dimension $2p$, Proc. Amer. Math. Soc. 133 (8) (2005), 2237--2242.
\bibitem[R75]{R75}D.-E. Radford, On the coradical of a finite-dimensional Hopf algebra, Proc. Amer. Math. Soc. 53
(1975), 9--15.
\bibitem[R85]{R85}D.-E.~Radford, The structure of Hopf algebras with a projection, J. Algebra
 92 (2) (1985), 322--347.
\bibitem[R11]{R11}---, Hopf Algebras, Knots and Everything 49, World Scientific, 2011.
%\bibitem[Ro98]{Ro98}M.~Rosso, Quantum groups and quantum shuffles, Invent. Math., 133 (2) (1998), 399--416.
%\bibitem[S99]{S99}D.~Stefan, Hopf algebras of low dimension, J. Algebra 211 (1999), 343--361.
%\bibitem[T00]{T00}M. Takeuchi, Survey of braided Hopf algebras, Contemp Math. 267 (2000), 301¨C-323.
\bibitem[U07]{U07}S.~Ufer, Triangular braidings and pointed Hopf algebras, J. Pure Applied Algebra 210 (2)
(2007), 307--320.
\bibitem[X18]{X18}R. Xiong, Finite-dimensional Hopf algebras over the smallest non-pointed basic Hopf algebra,  arXiv:1801.06205.
%\bibitem[Z94]{Z94}Y.~Zhu, Hopf algebras of prime dimension, Intern. Math. Res. Notes 1 (1994), 53--59.
\end{thebibliography}
\end{document}